\documentclass[a4paper]{article}
\usepackage{fullpage}
\usepackage{graphicx} % Required for inserting images
\usepackage{epsfig}
\usepackage{tikz}
\usepackage{tikz-cd}
\usepackage{graphicx}
\usepackage{amsmath,amssymb,amsfonts}
\usepackage{amsthm}
\usepackage{mathrsfs}
\usepackage{bm}
\usepackage{stmaryrd}
\usepackage{amsrefs}
\usepackage{enumitem} % to change style of the enumeration [label=(\roman*)]
\numberwithin{equation}{section}

\theoremstyle{plain}
\newtheorem{theorem}{Theorem}[section]
\newtheorem{proposition}[theorem]{Proposition}

\newtheorem{lemma}[theorem]{Lemma}

\theoremstyle{definition}
\newtheorem{definition}[theorem]{Definition}
\newtheorem{example}[theorem]{Example}

\theoremstyle{remark}
\newtheorem{remark}[theorem]{Remark}

\usepackage{hyperref}
\title{Uniform Diophantine approximation on the Hecke group $\mathbf H_4$}

\author{Ayreena Bakhtawar
\and
Dong Han Kim
\and
Seul Bee Lee
}
\newcommand{\Addresses}{{% additional braces for segregating \footnotesize
  \bigskip
  \footnotesize

  A.~Bakhtawar, \textsc{Centro di Ricerca Matematica Ennio De Giorgi, Scuola Normale Superiore, Piazza dei Cavalieri 3,
56126 Pisa, Italy and \\
Institute of Mathematics, Polish Academy of Sciences, ul.  Sniadeckich 8, 00-656
Warszawa, Poland}\par\nopagebreak
  \textit{E-mail address}: \texttt{ayreena.bakhtawar@sns.it,abakhtawar@impan.pl}

  \medskip

  D.H.~Kim, \textsc{Department of Mathematics Education, Dongguk University - Seoul, 30 Pildong-ro 1-gil, Jung-gu, Seoul, 04620 Korea}\par\nopagebreak
  \textit{E-mail address}: \texttt{kim2010@dgu.ac.kr}

  \medskip

  S.B.~Lee, \textsc{Department of Mathematical Sciences, Seoul National University, 1, Gwanak-ro, Gwanak-gu, Seoul 08826, Korea}\par\nopagebreak
  \textit{E-mail address}: \texttt{seulbee.lee@snu.ac.kr}

}
}

\date{\today}

\begin{document}

\maketitle

\begin{abstract}
Dirichlet's uniform approximation theorem is a fundamental result in Diophantine approximation that gives an optimal rate of approximation with a given bound.  %%d 
We study uniform Diophantine approximation properties on the Hecke group $\mathbf H_4$.
For a given real number $\alpha$, we characterize  %%d
the sequence of $\mathbf H_4$-best  %%d
approximations of $\alpha$ and show that they are convergents of the Rosen continued fraction and the dual Rosen continued fraction of $\alpha$.
We give analogous theorems of Dirichlet uniform approximation and the Legendre theorem with optimal constants.
\end{abstract}

\section{Introduction}

The celebrated Dirichlet's theorem states that for any real number $\alpha$ and for every $N>1$,
there exist %a rational number $p/q$, where %%d
$p,q\in\mathbb Z$ such that   
\begin{equation}\label{DT}
| q\alpha  - p | < \frac{1}{N}, \qquad 1 \le q \le N.
\end{equation}
Dirichlet’s theorem is a uniform Diophantine approximation result as it guarantees a non-trivial integer solution for all large enough $N$.
An irrational number $\alpha$ is said to be \emph{Dirichlet's improvable} if there exists a constant $C<1$ such that the inequalities   %%d
\begin{equation}\label{DTS}
|q \alpha  - p | < \frac{C}{N}, \qquad 1 \le q \le N,
\end{equation}
have  %%d
non trivial integer solutions for all sufficiently large enough $N$. 
It has been observed by Davenport and Schmidt \cite{DS70} that the set of Dirichlet's improvable numbers has Lebesgue measure zero. 
Further they showed that an irrational number $\alpha$ is Dirichlet's improvable if and only if it is badly approximable, i.e., it has bounded partial quotients in the continued fraction expansion of $\alpha.$
To be more precise, let $\alpha  = [ a_0; a_1, a_2, \dots]$ represents the regular continued fraction expansion of $\alpha$ and define $C(\alpha)$ as the infimum of $C$ satisfying \eqref{DTS}.
Then, it is known in \cite{DS70}*{Theorem 1} that
$$
C(\alpha) = \limsup_{n \to \infty} \, q_{n} | q_{n-1} \alpha - p_{n-1} |
= \limsup_{n \to \infty} \frac{1}{1+[ 0 ; a_n, a_{n-1}, \dots, a_1] \cdot [ 0; a_{n+1}, a_{n+2}, \dots]},
$$
where $p_n/q_n = [a_0; a_1, \dots, a_n]$ is the $n$-th convergent of the regular continued fraction expansion. 
We remark that $C(\alpha) < 1$ if and only if $\alpha$ is badly approximable.

Motivated by the work of Davenport and Schmidt \cite{DS70} for the regular continued fractions, 
the aim of this article is to study the uniform Diophantine approximation on the Hecke group $\mathbf H_4$
and extends the results of \cite{DS70} for $\mathbf H_4.$ 
It is worth mentioning here that, 
in the settings of asymptotic Diophantine approximation (i.e., when one considers the following statement: for any irrational number, the inequality $|\alpha x-y|<c/x$  has infinitely many solutions in $x,y\in\mathbb Z$ and $x\neq 0$ for some constant $c$), there have been developments for the Hecke groups, see for example \cites{KS, Le85, RS92, Vul97}, but nothing much is known for the uniform approximation.
%, i.e., improving Dirichlet's original theorem \ref{DT}.  %%d
The results of this paper thus will contribute to enhancing the complete description of $\mathbf H_4$-expansion in the settings of uniform Diophantine approximation.

\begin{figure}
\centering
\begin{tikzpicture}[scale=3]
    \node[below] at (0, 0) {$0$};
    \draw[thick] (-1.5, 0) -- (1.5, 0);
%    \draw[thick] ({-sqrt(2)},1.6) -- ({-sqrt(2)},0) node [below] {$-\sqrt 2$};
%    \draw[thick] (0,1.6) -- (0,0) node [below] {$0$};
%    \draw[thick] ({sqrt(2)},1.6) -- ({sqrt(2)},0) node [below] {$\sqrt 2$};
    \draw[thick,red] ({-1/sqrt(2)},1.6) -- ({-1/sqrt(2)},{1/sqrt(2)});
    \draw[dashed,red] ({-1/sqrt(2)},{1/sqrt(2)}) -- ({-1/sqrt(2)},0) node [below] {$-\frac1{\sqrt 2}$};
    \draw[thick,red] ({1/sqrt(2)},1.6) -- ({1/sqrt(2)},{1/sqrt(2)});
    \draw[dashed,red] ({1/sqrt(2)},{1/sqrt(2)}) -- ({1/sqrt(2)},0) node [below] {$\frac1{\sqrt 2}$};
    \draw[thick,blue] ({1/sqrt(2)},{1/sqrt(2)}) arc (45:135:1);
    \draw[dashed,blue] (1,0) node [below] {$1$} arc (0:45:1);
    \draw[dashed,blue] (-1,0) node [below] {$-1$} arc (180:135:1);
%    \draw[thick] ({-sqrt(2)},0) arc (180:0:{sqrt(2)/4});
%    \draw[thick] ({-1/sqrt(2)},0) arc (180:0:{sqrt(2)/4});
%    \draw[thick] (0,0) arc (180:0:{sqrt(2)/4});
%    \draw[thick] ({1/sqrt(2)},0) arc (180:0:{sqrt(2)/4});
\end{tikzpicture}
\caption{Fundamental domain of the Hecke group $\mathbf H_4$}
\label{fig1}
\end{figure}
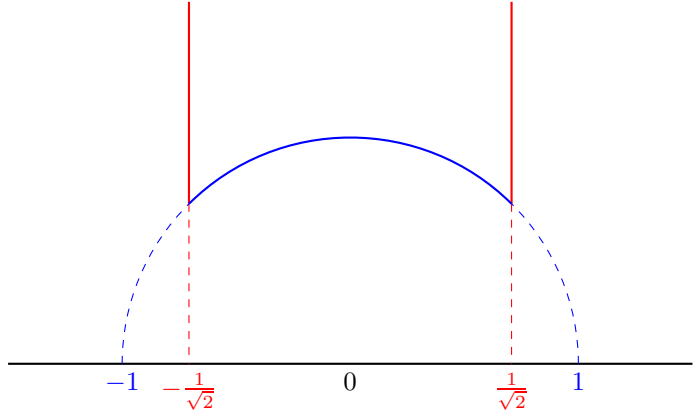

The Hecke group $\mathbf H_4$ is the subgroup of $\mathrm{PSL}_2(\mathbb R)$ generated by 
$\begin{pmatrix} 1 & \sqrt 2 \\ 0 & 1 \end{pmatrix}$ and   
$\begin{pmatrix} 0 & -1 \\ 1 & 0 \end{pmatrix}$.
See Figure~\ref{fig1} for a fundamental domain of $\mathbf H_4$.
It is well-known (see e.g. \cite{Pa77} and the references therein) that    %%d
\begin{multline*}
\mathbf H_4 = \left\{ \begin{pmatrix} a & \sqrt 2 b \\ \sqrt 2 c & d\end{pmatrix} \, \big| \ ad-2bc = 1, \quad a, d \in 2\mathbb Z +1 , \quad b,c \in \mathbb Z \right\} \\
\cup \left\{ \begin{pmatrix} \sqrt 2 a & b \\ c & \sqrt 2 d\end{pmatrix} \, \big| \ 2ad-bc = 1, \quad b,c \in 2\mathbb Z +1 , \quad a,d \in \mathbb Z \right\}.
\end{multline*}   %%d
The Hecke group $\mathbf H_4$ acts as a linear fractional map on the upper half plane $\mathbb H = \{ x + yi \in \mathbb C \, | \, y > 0 \}$ and its boundary $\partial \mathbb H = \hat{\mathbb R} := \mathbb R \cup \{ \infty \}$.
We have approximations of a real number by the image of $\infty$ under the action of $\mathbf H_4$.  
Let 
$$
\mathbb Q (\mathbf H_4) := 
%\{ M \cdot \infty \in \hat{\mathbb R} \, | \, M \in \mathbf H_4 \} =    %%d
\left\{ \frac pq \in \hat{\mathbb R} \ \Big | \, \begin{pmatrix} p & r \\ q & s \end{pmatrix} \in \mathbf H_4 \right \}
%$$
%Therefore, we have   
%$$
%\mathbb Q (\mathbf H_4) 
= \sqrt 2 \mathbb Q.   %%d
$$
We will write $p/q\in\mathbb Q (\mathbf H_4)$ if there exists $\begin{pmatrix}p& r\\q&s\end{pmatrix}\in\mathbf H_4$ for some $r$, $s$ with $q>0$, i.e.,
$$
p = a , \ q = \sqrt 2 c \quad \text{ or } \quad  p = \sqrt 2 a, \ q =c 
$$
for some coprime integers $a$ and $c > 0$.
We obtain the following main results on the uniform Diophantine approximation for the Hecke group $\mathbf H_4$.

\begin{theorem}\label{thm:DS}
Let $\alpha$ be a real number.
For every $N$, there exists $p/q \in \mathbb Q (\mathbf H_4)$ such that
\begin{equation}\label{thmua}
\left| q \alpha - p \right| < \frac{\sqrt 2+1}{2} \cdot \frac{1}{N} \quad \text{ with } \quad 1 \le q \le N.
\end{equation}
\end{theorem}

We consider improvement of the Dirichlet type theorem for the Hecke group $\mathbf H_4$.
Define $K(\alpha)$ as the infimum of $K$
such that the following inequality 
\begin{equation}\label{DTSRC}
\left| q \alpha  - p \right| < \frac{K}{N}, \qquad 1 \le q \le N,
\end{equation}
has non trivial solutions $\frac pq \in \mathbb Q (\mathbf H_4)$ for all sufficiently large enough $N$. 
Indeed, we showed that 
$$K(\alpha) \le K(1) = \frac{\sqrt 2 + 1}{2}$$ 
for any $\alpha \in \mathbb R \setminus \mathbb Q(\mathbf H_4)$ 
(see Example~\ref{example:sup} below). 
Therefore the constant $\frac{\sqrt 2 + 1}{2}$ in Theorem~\ref{thm:DS} cannot be replaced by any smaller value. 
We remark that $1$ is badly approximable but not Dirichlet's improvable in $\mathbf H_4$.   %%d
Indeed, $1$ is the number corresponding to the minimum of the Lagrange spectrum and   %%d
$$\liminf_{p/q \in \mathbb Q (\mathbf H_4)} \, q \left| q \cdot 1  - p \right| = \frac 12.$$
See \cites{CK23, KS} for details.

The uniform approximation constant $K(\alpha)$ is obtained by the sequence of best approximations.
Let us first recall the definition of best approximation.
\begin{definition}
Let $p/q \in \mathbb Q (\mathbf H_4).$ Then $p/q$ is an $\mathbf H_4$-best approximation of a real number $\alpha$ if for any $r/s \in \mathbb Q (\mathbf H_4)$ with $r/s \ne p/q$ and $0 < s \le q$,
$$
|s\alpha - r| > | q\alpha - p|.
$$
\end{definition}
Let
$\Big \{ \frac{p_i}{q_i} \Big\}_{i=1}^\infty $
be the sequence of the $\mathbf H_4$-best approximations with $q_i < q_{i+1}$.
Similarly to the classical modular group case, we have
$$
K(\alpha) = \limsup_{i \to \infty} q_i | q_{i-1} \alpha - p_{i-1}|,
$$
(see Lemma~\ref{lem:K} below).
We give a bound of the $\mathbf H_4$-best approximation and state Legendre's theorem associated with $\mathbf H_4$-best approximations.
\begin{theorem}\label{thm2}
(i) If $p/q$ is an $\mathbf H_4$-best approximation of $\alpha$, then 
we have 
$$
\left| \alpha - \frac{p}{q} \right| < \frac{1}{q^2}.
$$
(ii) If 
\begin{equation*}
\left|\alpha-\frac{p}{q}\right|<\frac{1}{2q^2},
\end{equation*}
then $p/q$ is an $\mathbf H_4$-best approximation of $\alpha$.
Moreover, the constant 1/2 is the optimal value.
\end{theorem}
Note that for the case $p/q$ 
is a convergent of the Rosen continued fraction,  %%d
Theorem~\ref{thm2} (i) was shown in \cite{Na95}*{Proposition 12}.

Our next result concerns about the approximation properties of the convergents of an irrational number $\alpha$ associated with the Rosen continued fraction expansion.
The \emph{Rosen continued fraction} of a real number $\alpha$ associated with the Hecke group $\mathbf H_4$ can be defined as follows:
\begin{equation*}
\alpha = a_0 \sqrt 2 + \cfrac{\epsilon_1}{ a_{1} \sqrt 2 + \cfrac{\epsilon_2}{ a_{2} \sqrt 2+ \ddots }}
= \left \llbracket a_0 ; \epsilon_1/a_1,  \epsilon_2/a_2, \dots \right \rrbracket,
\end{equation*}
here, we put $k_0 \in \mathbb Z$ and $k_i \in \mathbb N$, $\epsilon_i \in \{ -1, +1 \}$ for $i \ge 1$. 
Let 
$$ 
\frac{r_i}{s_i} = \llbracket a_0; \epsilon_1 / a_1, \dots, \epsilon_i/ a_i \rrbracket =
a_0 \sqrt 2 + \cfrac{\epsilon_1}{ a_{1} \sqrt 2 + \cfrac{\epsilon_2}{ \ddots+ \cfrac{\epsilon_{i}}{a_i \sqrt 2} }}
$$
be the $i$-th convergent of the Rosen continued fraction.
In addition, for the uniqueness of the expansion, we assume that if $a_i = 1$  for $i \ge 1$, then $\epsilon_{i+1} = +1$.
%In order to differentiate the representation of Rosen continued fraction expansion from the representation of regular continued fraction expansion, we have used the notion $\llbracket \cdot\rrbracket.$   %%d
One can change the constraint for the uniqueness of the Rosen continued fraction expansion. 
The \emph{dual Rosen continued fraction} expansion is defined by
$$\alpha 
= \big \llbracket \tilde{a}_0 ; \tilde{\epsilon}_1/\tilde{a}_1 , \tilde{\epsilon}_2/\tilde{a}_2, \dots \big \rrbracket,
$$
where $\tilde a_0 \in \mathbb Z$ and $\tilde a_i \in \mathbb N$, $\tilde\epsilon_i \in \{ -1, +1 \}$ for $i \ge 1$ with the rule that if $\tilde{a}_i = 1$, then $\tilde{\epsilon}_{i} = +1$ for $i \ge 1$.

We can check easily that the first best approximation 
$\frac{p_1}{q_1} = a_0\sqrt 2 = \frac{r_0}{s_0}$
is the 0-th convergent of the Rosen continued fraction.
The convergents of the Rosen and the dual Rosen  %%d
continued fraction expansion are $\mathbf H_4$-best approximations of an irrational number in the following sense:
\begin{theorem}\label{thm:BA}
Let $\alpha$ be a real number not in $\mathbb Q (\mathbf H_4).$ 
Let $\frac{r_i}{s_i} = \llbracket a_0; \epsilon_1 / a_1, \dots, \epsilon_i/ a_i \rrbracket$ 
be the $i$-th convergent of the Rosen continued fraction and 
$\frac{\tilde r_i}{\tilde s_i} = \llbracket \tilde a_0; \tilde \epsilon_1 / \tilde a_1, \dots, \tilde \epsilon_i/ \tilde a_i \rrbracket$
be the $i$-th convergent of the dual Rosen continued fraction.
We have that $p/q \in \mathbb Q (\mathbf H_4)$ is an $\mathbf H_4$-best approximation of $\alpha$ if and only if 
$$ \frac{p}{q} = \frac{r_i}{s_i} 
\quad \text{ for } \ i \ge 0 \quad \text{ or } \quad  \frac pq = \frac{\tilde r_i}{\tilde s_i} 
\quad \text{ for } \ i \ge 1.$$
\end{theorem}

The structure of the paper is as follows. 
In Section~\ref{sec:AR}, we review the Hecke group and give the $\mathbf H_4$-expansion of real numbers. 
In Section~\ref{sec:BA}, we characterize the $\mathbf H_4$-best approximations. 
In Section~\ref{sec:DS}, the proof of Theorem~\ref{thm:DS} is given. 
We discuss relation between the best approximation of the Rosen continued fraction and give the detail of the proof of Theorem~\ref{thm:BA} in Section~\ref{sec:RCF}.
Then, we discuss properties of the $\mathbf H_4$-best approximations and prove Theorem~\ref{thm2} in Section~\ref{sec:LT}.  %%d

\section{\texorpdfstring{$\mathbf H_4$}{H4}-expansion of a real number}\label{sec:AR}

The Hecke (triangle Fuchsian) group $\mathbf H_{k}$ with an integer $k\ge 3$ is the subgroup of $\mathrm{PSL}_2(\mathbb R)$ generated by 
\begin{equation*}
T = \begin{pmatrix} 1 & \lambda_{k} \\ 0 & 1 \end{pmatrix} \quad \text{ and } \quad 
S = \begin{pmatrix} 0 & -1 \\ 1 & 0 \end{pmatrix},
\end{equation*}
where $\lambda_{k}=2\cos\left(\frac{\pi}{k}\right)$.
If $k = 3$, then $\lambda_3 = 1$ and $\mathbf H_3$ is the modular group $\mathrm{PSL}_2(\mathbb Z)$  
and $\mathbb Q (\mathbf H_3) = \mathbb Q$.  %%d
When $k=4$, then we have $\lambda_{4}=\sqrt{2}$  
and %it is well-known (see e.g. \cite{Pa77}) that 
$\mathbb Q (\mathbf H_4) = \sqrt 2 \mathbb Q$.   %%d

In 1954, Rosen \cite{Ro54} introduced a class of continued fractions known as Rosen continued fractions in order to study the Hecke groups. 
The Rosen continued fraction expansion of a real number $\alpha$ terminates at a
finite term if and only if $\alpha$ is a parabolic fixed point of $\mathbf H_{k}$, see \cite{Ro54}. 
These points are clearly contained in the extension field $\mathbb{Q}(\lambda_{k})$ but in general, there are elements of
this field that have an infinite Rosen continued fraction expansion, see \cite{HMTY}.
Introduced to aid in the study of certain Fuchsian groups, these continued fractions were applied some thirty years later by J. Lehner \cite{Le85} in the study of asymptotic Diophantine approximation by orbits of these groups.
Following that, a rich theory on the Rosen continued fraction has emerged including studies on asymptotic Diophantine approximation, metrical theory associated to Rosen continued fractions and applications to dynamical systems, for details we refer the reader to \cites{Na10, RS92, SW16I}.
Since the aim of this article is to study Diophantine approximation on $\mathbf H_4$, we will therefore consider $k=4$ throughout the paper.

\begin{figure}
\centering
\begin{tikzpicture}[scale=4]
	\draw[thick] (0, 0) -- ({sqrt(2)}, 0);
	\draw[thick] (0,1.2) -- (0,1);
	\draw[thick] ({sqrt(2)},1.2) -- ({sqrt(2)},1);

    \node[below] at ({1/sqrt(2)},0) {$\frac{1}{\sqrt 2}$};
    \node[below] at (0,0) {$0$};
    \node[below] at ({sqrt(2)},0) {$\sqrt 2$};
    \draw[thick] (0,1) arc (90:45:1);
    \draw[thick] ({sqrt(2)},1) arc (90:135:1);
    \draw[thick,dashed,blue] ({1/sqrt(2)},1.2) -- ({1/sqrt(2)},{1/sqrt(2)});
%    \draw[thick,dashed,blue] (0,0) arc (180:0:{1/sqrt(2)});
  
 %   \draw[thick] (0,0) arc (180:0:{1/(2*sqrt(2))});
    %node [below] {$\frac{1}{\sqrt 2}$};
%    \draw[thick] ({sqrt(2)},0) arc (0:180:{1/(2*sqrt(2))}); 
\end{tikzpicture}
\qquad
\begin{tikzpicture}[scale=4]
	\draw[thick] (0, 0) -- ({sqrt(2)}, 0);
	\draw[thick] (0,1.2) -- (0,0) node [below] {$0$};
	\draw[thick] ({sqrt(2)},1.2) -- ({sqrt(2)},0) node [below] {$\sqrt 2$};

    \node[below] at ({1/sqrt(2)},0) {$\frac{1}{\sqrt 2}$};
    \draw[thick,red] (0,1) arc (90:19.47:1);
    \draw[thick,red] ({sqrt(2)},1) arc (90:160.53:1);
    \draw[thick,dashed,blue] ({1/sqrt(2)},1.2) -- ({1/sqrt(2)},0);
    \draw[thick,dashed,blue] (0,0) arc (180:0:{1/sqrt(2)});
  
    \draw[thick] (0,0) arc (180:0:{1/(2*sqrt(2))});
    %node [below] {$\frac{1}{\sqrt 2}$};
    \draw[thick] ({sqrt(2)},0) arc (0:180:{1/(2*sqrt(2))}); 
    \node[left] at (0,.8) {$\bm\delta_0$};
    \node[right] at ({sqrt(2)},.8) {$\bm\delta_3$};
    \node[below] at ({1/(2*sqrt(2))},.34) {$\bm\delta_1$};
    \node[below] at ({3/(2*sqrt(2))},.34) {$\bm\delta_2$};

\end{tikzpicture}

\caption{The fundamental domain $\bm\Omega$ of $\mathbb H_4$ (left) and the ideal quadrilateral $\bm\Delta$ (right)}
\label{fig:domain}
\end{figure}
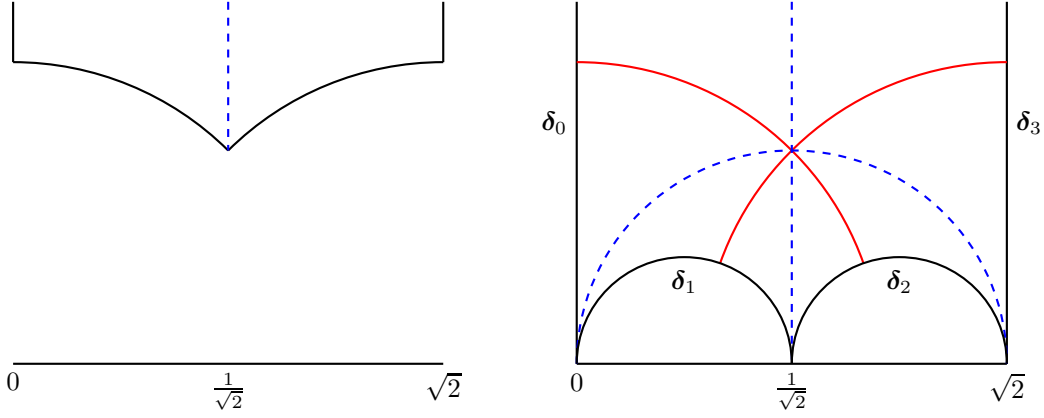

Let $\bm\Omega$ be another fundamental domain of $\mathbf H_4$ surrounded by geodesics $x = 0$, $x = \sqrt 2$, $|z| = 1$, and $|z - \sqrt 2 | = 1$ (see Figure~\ref{fig:domain} (left)).
Let 
$$
R = ST^{-1} = \begin{pmatrix} 0 & 1 \\ -1 & \sqrt 2 \end{pmatrix}.
$$
Then $R^4 = I$.
Let $\bm\Delta = \bm\Omega \cup R(\bm\Omega) \cup R^2(\bm\Omega) \cup R^3(\bm\Omega)$. 
Let $\bm\delta_0$ be the geodesic on $\mathbb H$ connecting boundary points $\infty$ and $0$ on $\hat{\mathbb R}$
and put $\bm\delta_i = R^i (\bm\delta_0)$ for $i = 1,2,3$.
Then the ideal quadrilateral $\bm\Delta$ is surrounded by $\bm\delta_i$, $i = 0,1,2,3$ (see Figure~\ref{fig:domain} (right)). 
Let $K_i = R^i SR^{-i}$ for $i = 0,1,2,3$ 
and $\bm\Gamma_4$ be the group generated by $K_0, K_1, K_2, K_3$.
Then $\bm\Delta$ is the fundamental domain of $\bm\Gamma_4$
and $\bm\Gamma_4$ is an index 4-subgroup of $\mathbf H_4$ (see \cite{HS86}). 
For any element $G \in \mathbf H_4$, there exists $H \in \bm\Gamma_4$ and $d \in \{ 0,1,2,3\}$ such that $G = H R^d$.
Denote
\begin{equation*}\label{def:M}
A_1 = RS = \begin{pmatrix} 1 & 0 \\ \sqrt 2 & 1 \end{pmatrix},
\quad 
A_2 = R^2S = \begin{pmatrix} \sqrt 2 & 1 \\ 1 & \sqrt 2 \end{pmatrix}, \quad
A_3 = R^3S = \begin{pmatrix} 1 & \sqrt 2 \\ 0 & 1 \end{pmatrix}. 
\end{equation*}
Then,  %%d 
for any element $G \in \mathbf H_4$, there exist $d_1, \dots, d_k \in \{ 1,2,3\}$ and $d \in \{ 0,1,2,3\}$ such that 
$G = A_{d_1} \cdots A_{d_k} R^d$
or $G = S A_{d_1} \cdots A_{d_k} R^d$.

We follow the symbolic coding of a geodesic on the  developed by Haas and Series \cite{HS86} and Series \cite{Ser88}
(see also \cites{KS}). %%d
Let $\bm\gamma$ be an oriented geodesic with end points 
$\bm\gamma^- \in (-\infty,0)$, $\bm\gamma^+ \in (0, \infty)$ and assume further that $\bm\gamma^-$ and $\bm\gamma^+$ does not belong to $\mathbb Q(\mathbf H_4)$.
Note that $\bm\gamma$ intersects $\bm\delta_0$ while passing from $S(\bm\Delta)$ to $\bm\Delta$.
By cutting $\mathscr T = \cup_{G \in \bm\Gamma_4} G(\partial \bm\Delta) = \cup_{G \in \mathbf H_4} G(\bm\delta_0)$, we have a sequence of geodesic segments 
$\dots, \bm\gamma_{-1}, \bm\gamma_0, \bm\gamma_{1}, \bm\gamma_{2}, \bm\gamma_{3}, \dots$ 
with end points $\bm\gamma_n^-, \bm\gamma_n^+ \in \mathscr T$ along the orientation of $\bm\gamma$,
satisfying that $\bm\gamma_n^+ = \bm\gamma_{n+1}^- \in \mathscr T$ for $n \in \mathbb Z$.
We choose $\bm\gamma_0$ to belong to $\bm\Delta$.
Then $\bm\gamma_0^- \in \bm\delta_0$ and there exists $d_1 \in \{ 1,2,3\}$ such that $\bm\gamma_0^+ \in \bm\delta_{d_1}$.
For each $n \in \mathbb N$, let $G_n \in \mathbf H_4$ be such that $\bm\gamma_n$ belongs to $G_{n}(\bm\Delta)$ and $\bm\gamma_n^- \in G_n(\bm\delta_0)$.
Choose $d_{n+1} \in \{1,2,3\}$ as $\bm\gamma_{n}^+ \in G_{n} (\bm\delta_{d_{n+1}})$.
Then by $\bm\gamma_{n+1}^- = \bm\gamma_{n}^+ \in G_{n} (\bm\delta_{d_{n+1}}) = G_{n} R^{d_{n+1}} (\bm\delta_0) = G_{n} R^{d_{n+1}} S(\bm\delta_0)$,
we have
$$G_{n+1} = G_{n} R^{d_{n+1}}S = G_{n}A_{d_{n+1}}.$$
We have for all $n \ge 1$ 
\begin{equation*}\label{GM}
G_{n} = A_{d_1} A_{d_2} \cdots A_{d_n}. 
\end{equation*}
Since $\bm\gamma$ crosses $G_{n} ( S (\bm\Delta))$ and then intersects $ G_{n}(\bm\Delta)$,
we have for all $n \ge 1$
\begin{equation*}\label{endpoints}
\bm\gamma^+ \in G_{n} \cdot (0, \infty) = A_{d_1} A_{d_2} \cdots A_{d_n} \cdot (0, \infty).
\end{equation*}
Let $\alpha = \bm\gamma^+$ be a positive real number.
We write
$$
\alpha = [ d_{1}, d_{2}, d_3, \dots ]
$$
and call it the \emph{$\mathbf H_4$-expansion} of $\alpha$. 
Check  
$$
[0, \infty] = \left[ 0,\frac{1}{\sqrt 2} \right] \cup \left[\frac{1}{\sqrt 2} , \sqrt 2 \right] \cup \left[\sqrt 2 , \infty \right] =
A_1 \cdot [0,\infty] \cup A_2 \cdot [0,\infty] \cup A_3 \cdot [0,\infty].
$$
Therefore, for each $n \ge 1$, we have
$$
[0, \infty] = \bigcup_{(d_1,\dots,d_n)\in \{1,2,3\}^n } A_{d_1} \cdots A_{d_n} \cdot [0,\infty].
$$
For any $\alpha \in [0,\infty]$, we have a sequence $(d_n)_{n \in \mathbb N} \in \{ 1,2,3\}^{\mathbb N}$  satisfying that
\begin{equation*}\label{A_n}
\alpha \in A_{d_1} \cdots A_{d_n} \cdot [0,\infty] \qquad \text{ for all } \  n \ge 1.
\end{equation*}
Then we write
$$
\alpha = [d_1, d_2, d_3, \dots ].
$$
%%d
Let 
$$
\alpha_n = [d_{n+1}, d_{n+2}, d_{n+3}, \dots ].
$$    %%d
Then, it's directly deduced that 
$$
\alpha %= [d_1, d_2, d_3, \dots ] 
%= A_{d_1} A_{d_2} \cdots A_{d_n} \cdot [d_{n+1}, d_{n+2}, d_{n+3}, \dots ]   %%d
= A_{d_1} A_{d_2} \cdots A_{d_n} \cdot \alpha_n = G_n \cdot \alpha_n \in G_n \cdot [0,\infty].
$$   %%d

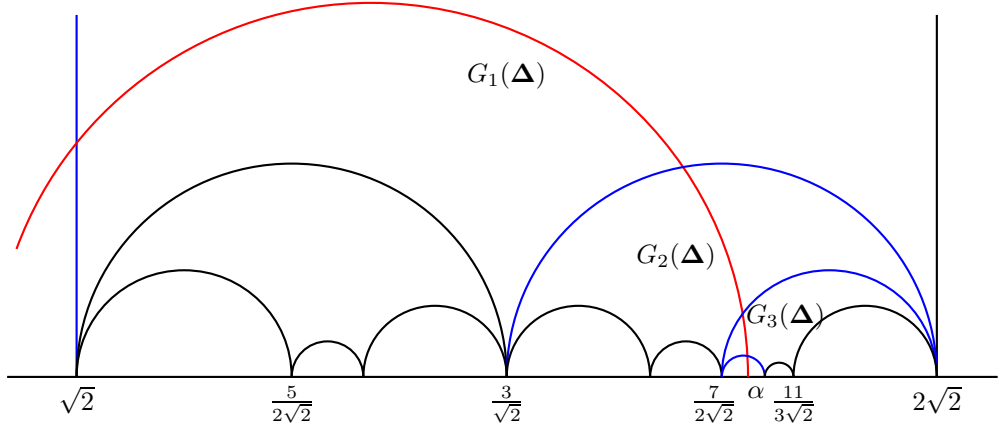
\begin{figure}
\centering
\begin{tikzpicture}[scale=8]
	\draw[thick] (1.3, 0) -- ({2*sqrt(2)+.1}, 0);
	\draw[thick, blue] ({sqrt(2)},0) -- ({sqrt(2)},.6);
	\draw[thick] ({2*sqrt(2)},.6) -- ({2*sqrt(2)},0); 
    \draw[thick, red] (2.518, 0)  arc (0:160:.62);
    \node at (2.518, 0) [below, xshift=3pt] {$\alpha$};

    \node[below] at ({sqrt(2)},0) {$\sqrt 2$};
    \node[below] at ({2*sqrt(2)},0) {$2\sqrt 2$};
    \node[below] at ({3/sqrt(2)},0) {$\frac{3}{\sqrt 2}$};
    \node[below] at ({5/(2*sqrt(2))},0) {$\frac{5}{2\sqrt 2}$};
    \node[below, xshift=-3pt] at ({7/(2*sqrt(2))},0) {$\frac{7}{2\sqrt 2}$};
    \node[below] at ({11/(3*sqrt(2))},0) {$\frac{11}{3\sqrt 2}$};

    \node at ({3/sqrt(2)},.5) {$G_1(\bm\Delta)$};   %%d
    \node at (2.4,.2) {$G_2(\bm\Delta)$};   %%d
    \node at (2.58,.10) {$G_3(\bm\Delta)$};   %%d
    
    \draw[thick] ({sqrt(2)},0) arc (180:0:{1/(2*sqrt(2))});
    \draw[thick, blue] ({2*sqrt(2)},0) 
    arc (0:180:{1/(2*sqrt(2))});

    \draw[thick] ({sqrt(2)},0) arc (180:0:{1/(4*sqrt(2))});
    \draw[thick, blue] ({2*sqrt(2)},0) 
    arc (0:180:{1/(4*sqrt(2))});
    
    \draw[thick] ({3/sqrt(2)},0) arc (0:180:{1/(6*sqrt(2))});
    \draw[thick] ({3/sqrt(2)},0) arc (180:0:{1/(6*sqrt(2))});

    \draw[thick] ({5/(2*sqrt(2))},0) arc (180:0:{1/(12*sqrt(2))});
    \draw[thick] ({7/(2*sqrt(2))},0) arc (0:180:{1/(12*sqrt(2))});
    
    \draw[thick] ({2*sqrt(2)},0) arc (0:180:{1/(6*sqrt(2))});
    \draw[thick, blue] ({7/(2*sqrt(2))},0) arc (180:0:{1/(20*sqrt(2))}); 
    \draw[thick] ({11/(3*sqrt(2))},0) arc (0:180:{1/(30*sqrt(2))});
    
\end{tikzpicture}
\caption{Cutting sequence along the geodesic from $\infty$ to $\alpha = [3,2,3,1,\dots]$}
\label{fig:cuttingsq}
\end{figure}

Note that 
$$
0 = [ 1, 1, 1, \dots] = [1^\infty], \quad 
1 = [ 2, 2, 2, \dots] = [2^\infty], \quad
\infty = [ 3, 3, 3, \dots] = [3^\infty].
$$
If $\alpha \in \mathbb Q(\mathbf H_4) \cap (0,\infty)$, then there are exactly two such sequences. 
For examples, we have 
$$
[1, 3^\infty] = A_1 \cdot \infty = \frac{1}{\sqrt 2} = A_2 \cdot 0 = [ 2, 1^\infty], \quad
[2, 3^\infty] =  A_2 \cdot \infty = \sqrt 2 = A_3 \cdot 0 = [ 3, 1^\infty].
$$

\begin{example}
%Since $1 = A_2 \cdot 1$, 
%we have $\alpha 1 = [2,2,2,\dots]$.
%Then $G_n = A_2^n$ and $\alpha_{n} = [2,2,2,\dots] = 1$ for all $n \ge 1$.
%Note that 
%$ 1 =\llbracket 1 ; -1/2, -1/2, \dots \rrbracket$.
%
For $\alpha = \frac{3 + \sqrt{17}}{2\sqrt{2}}$
we check that 
$$
\alpha = [(323121)^\infty] = [3,2,3,1,2,1,3, \dots ].
$$
Note that $G_1= A_3$, $G_2 = A_3 A_2$, $G_3 = A_3 A_2 A_3$, $G_4 = A_3 A_2 A_3 A_1$ and   %%d
%= \begin{pmatrix} 1 & \sqrt 2 \\ 0 & 1 \end{pmatrix}, \
%= \begin{pmatrix} 2\sqrt 2 & 3  \\ 1 & \sqrt 2 \end{pmatrix}, \
%= \begin{pmatrix} 2\sqrt 2 & 7  \\ 1 & 2\sqrt 2 \end{pmatrix}, \\ %\begin{pmatrix} 9\sqrt 2 & 7  \\ 5 & 2\sqrt 2 \end{pmatrix}, 
%\begin{pmatrix} 25 & 16\sqrt 2 \\ 9\sqrt 2 & 9 \end{pmatrix}, 
%\begin{pmatrix} 57 & 16\sqrt 2 \\ 16\sqrt 2 & 9 \end{pmatrix}, \dots  %%d
\begin{align*}   %%d
\alpha &\in G_1 \cdot (0,\infty) = \left[ \sqrt 2 , \infty  \right], & 
\alpha &\in G_2 \cdot (0,\infty) = \left[ \frac{3}{\sqrt 2}, 2\sqrt2 \right], \\
\alpha &\in G_3 \cdot (0,\infty) = \left[ \frac{7}{2\sqrt 2}, 2\sqrt2 \right], &  
\alpha &\in G_4 \cdot (0,\infty) = \left[ \frac{7}{2\sqrt 2}, \frac{9\sqrt2}{5} \right].
\end{align*}   %%d
See Figure~\ref{fig:cuttingsq}. 
%for the picture of $G_n(\bm\Delta)$ and the geodesic to $\alpha$.
For the numbers of periodic $\mathbf H_4$-expansion, consult \cite{CK20}.
\end{example}

\section{\texorpdfstring{$\mathbf H_4$}{H4}-best approximations}
\label{sec:BA}

\begin{figure}
\centering
\begin{tikzpicture}[scale=2.6]
%	\node[below left] at (0, 0) {$O$};

	\draw[thick] (-.6, 0) -- (4.6, 0);

	\draw[ultra thick] (-.6,1.4142) -- (4.6,1.4142); 
	\draw[ultra thick] (-0,0) node [below] {$\frac{0}{1}$} arc (-90:270:0.7071); 
	\draw[ultra thick] (2,0) node [below ] {$\frac{\sqrt 2}{1}$} arc (-90:270:0.70711); 
	\draw[ultra thick] (4,0) node [below] {$\frac{2\sqrt 2}{1}$} arc (-90:270:0.7071); 

	\draw (0,1.6) -- (0,0);
	\draw (2,1.6) -- (2,0);
	\draw (4,1.6) -- (4,0);

	\draw[ultra thick] (1,0) node [below] {$\frac1{\sqrt 2}$} arc (-90:270:0.35355); 
	\draw[ultra thick] (3,0) node [below] {$\frac{3}{\sqrt 2}$} arc (-90:270:0.35355); 

    \draw (0,0) arc (180:0:1/2);
    \draw (1,0) arc (180:0:1/2);
    \draw (2,0) arc (180:0:1/2);
    \draw (3,0) arc (180:0:1/2);
    
	\draw[ultra thick] (1/2,0) node [below left] {\small $\frac{1}{2\sqrt 2}$} arc (-90:270:0.08839); 
	\draw[ultra thick] (2/3,0) node [below] {\small $\frac{\sqrt 2}{3}$} arc (-90:270:0.07857); 
	
    \draw (0,0) arc (180:0:1/4);
    \draw (1/2,0) arc (180:0:1/12);
    \draw (2/3,0) arc (180:0:1/6);

	\draw[ultra thick] (4/3,0) node [below] {\small $\frac{2\sqrt 2}{3}$} arc (-90:270:0.07857); 
	\draw[ultra thick] (3/2,0) node [below right] {\small $\frac{3}{2\sqrt 2}$} arc (-90:270:0.08839); 

    \draw (1,0) arc (180:0:1/6);
    \draw (4/3,0) arc (180:0:1/12);
    \draw (3/2,0) arc (180:0:1/4);
\end{tikzpicture}
\caption{Ford circles for the Hecke group $\mathbf H_4$}
\label{fig:FC}
\end{figure}
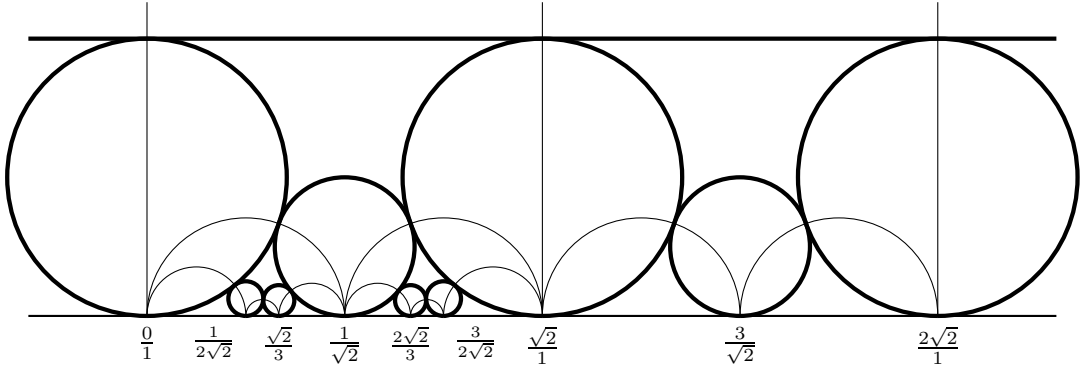

In this section, we characterize the set of $\mathbf H_4$-best approximations by using $\mathbf H_4$-expansions.
We assume that $q$ is non-negative for $p/q\in \mathbb Q(\mathbf H_4)$.

\emph{The Ford circles} are circles on the upper half-plane, which are tangent to the real line at the base point at $p/q \in \mathbb Q (\mathbf H_4)$ with Euclidean radius $1/(2q^2)$, see Figure~\ref{fig:FC}.
For example, the radius of the Ford circle based at $0/1$ is 1/2 and the radius of the Ford circle based at $1/\sqrt 2$ is 1/4.
Two Ford circles are either tangent or disjoint. 
Indeed, two Ford circles based at $p/q$ and $r/s$ are tangent if and only if 
\begin{equation*}
| ps - rq| = 1.
\end{equation*}
Note that for distinct %%d
$p/q, r/s \in \mathbb Q (\mathbf H_4)$, we have
$ps - rq  \in \mathbb Z \setminus \{0\}$ or $ps - rq  \in \sqrt 2 \mathbb Z \setminus \{0\}$. %%d
Therefore, we have 
$| ps - rq| \ge 1.$
We note that the Ford circles are invariant under the action of $\mathbf H_4$.

\begin{lemma}\label{lem1}
Let $a/c$, $b/d \in \mathbb Q(\mathbf H_4)$ be numbers in $\mathbb Q (\mathbf H_4)$ such that $| ad - bc| = 1$, equivalently the Ford circles based at $a/c$ and $b/d$ are tangent. 
If $p/q \in \mathbb Q(\mathbf H_4)$ is between $a/c$ and $b/d$, then we have
$$ q \ge c + d.$$ 
\end{lemma}

\begin{proof}
Without loss of generality, let us assume that $a/c>b/d$.
We prove it directly from the following calculation:
\begin{equation*}
\frac{1}{cd} = \frac ac - \frac{b}{d} = \frac ac - \frac pq + \frac pq - \frac bd 
\ge \frac{1}{cq} + \frac{1}{dq} = \frac{d+c}{cdq}.
\qedhere
\end{equation*}
\end{proof}

Let $\alpha = [d_1, d_2, \dots] \in (0,\infty) \setminus \mathbb Q(\mathbf H_4)$.  %%d
Denote
$$\begin{pmatrix} t_n & v_n \\ u_n & w_n \end{pmatrix} := G_n = A_{d_1} A_{d_2} \cdots A_{d_n}
$$ 
with the convention that $G_0=I$.
Then we have 
\begin{equation}\label{Galphan}
\alpha = G_n \cdot \alpha_n 
\in G_n \cdot ( 0,\infty) = \left( \frac{v_n}{w_n}, \, \frac{t_n}{u_n} \right).  %%d
\end{equation}
Note that 
%$$\alpha \in \left[ \frac{v_n}{w_n}, \, \frac{t_n}{u_n} \right]$$  %%d
%and  %%d
$$
\frac{v_n}{w_n} = [d_{1}, \dots , d_{n}, 1^\infty ], \qquad
\frac{t_n}{u_n} = [d_{1}, \dots , d_{n}, 3^\infty ]. 
$$

\begin{lemma}\label{lem3}
Let $\alpha$ be a positive real number not in $\mathbb Q(\mathbf H_4)$ and $p/q\in \mathbb Q(\mathbf H_4)$.
If 
$$
\frac pq \notin \left[ \frac{v_n}{w_n}, \, \frac{t_n}{u_n} \right] = 
G_n \cdot [0,\infty],
$$
then
$$ |q\alpha - p| > | u_n\alpha - t_n| \quad \text{ and } \quad |q\alpha - p| > | w_n\alpha - v_n|.$$ 
In particular, if $\frac pq \notin [ \frac{v_n}{w_n}, \, \frac{t_n}{u_n} ]$ 
and $q \ge \min\{ u_n, w_n\}$, 
then $p/q$ is not an $\mathbf H_4$-best approximation of $\alpha$.
\end{lemma}

\begin{proof}
Suppose that $v_n/w_n < \alpha < t_n/u_n < p/q$.
Since
$
t_nw_n - v_nu_n = 1 \le u_n p - t_n q, 
$
we deduce that 
$$
t_n(w_n+q) \le u_n(p+v_n).
$$
Therefore,
$$
\alpha < \frac{t_n}{u_n} \le \frac{p+v_n}{q+w_n}. 
$$
Thus we have
$$
0 < w_n \alpha-v_n < p - q\alpha.
$$

Also since 
$t_nw_n - v_nu_n = 1 \le p w_n - q v_n,$
if $u_n > q$, then
we have 
$$
\frac{t_n - p}{u_n - q} \le \frac {v_n}{w_n} < \alpha,
\quad \text{ thus } \quad
t_n - u_n\alpha < p - q\alpha.
$$
If $u_n < q$, then we have
$$
\alpha < \frac{t_n}{u_n} < \frac{p-t_n}{q-u_n}, \quad \text{ thus }
\quad
t_n - u_n\alpha < p - q\alpha. 
$$
If $u_n = q$, then $t_n < p$, thus 
$$ t_n - u_n\alpha < p - q\alpha.
$$
The proof for the case that $p/q < v_n/w_n < \alpha < t_n/u_n$ is similar.
\end{proof}

\begin{lemma}\label{lem:BA2}
Let $\alpha = [d_1, d_2, \dots]$ be a real number not in $\mathbb Q (\mathbf H_4)$.
Let $m \ge 0$ be the integer such that $d_1 = \dots = d_m = 3$ and $d_{m+1} \ne 3$.
An $\mathbf H_4$-best approximation is $t_n/u_n$ or $v_n/w_n$ for some $n \ge m+1$.
\end{lemma}

\begin{proof}
For $n \le m$, we have $u_n = 0$.
If $d_{m+1} = 1$, then $u_{m+1} = \sqrt 2$, $w_{m+1} = 1$.
If $d_{m+1} = 2$, then $u_{m+1} = 1$, $w_{m+1} = \sqrt 2$. 
Therefore, $\min\{ u_{m+1}, w_{m+1} \} = 1$.

For all $i\ge1$, we have
\begin{equation}\label{uw}
u_{i+1} =
\begin{cases}
u_i + \sqrt 2 w_i, &\text{ if } \ d_{i+1} = 1,\\
\sqrt 2 u_i + w_i, &\text{ if } \ d_{i+1} = 2,\\
u_i, &\text{ if } \ d_{i+1} = 3,
\end{cases}
\qquad 
w_{i+1} =
\begin{cases}
w_i, &\text{ if } \ d_{i+1} = 1,\\
u_i + \sqrt 2 w_i, &\text{ if } \ d_{i+1} = 2,\\
\sqrt 2 u_i + w_i, &\text{ if } \ d_{i+1} = 3.
\end{cases}
\end{equation}
Let $p/q \in \mathbb Q (\mathbf H_4)$ be such that $p/q\not=t_n/u_n$ and $p/q\not=v_n/w_n$ for all $n \ge m+ 1$.
By \eqref{uw}, $\min\{u_i,w_i\}$ is increasing.
Then there exists $n \ge m+ 1$ such that
\begin{equation*}\label{bbb1}
\min\{ u_n, w_n \} \le q < \min\{ u_{n+1}, w_{n+1} \}.
\end{equation*}
By using \eqref{uw}, we have  
\begin{equation}\label{bbb2}
q < \min\{ u_{n+1}, w_{n+1} \} \le 
\min \{ \sqrt 2 u_n + w_n,  u_n + \sqrt 2 w_n \}.
\end{equation}
Therefore, Lemma~\ref{lem1} implies that 
$$ \frac pq \notin \left(\frac{v_n}{w_n}, \frac{t_n + \sqrt 2  v_n}{u_n + \sqrt 2  w_n} \right) \cup  
\left(\frac{t_n + \sqrt 2  v_n}{u_n + \sqrt 2  w_n}, \frac{\sqrt 2 t_n + v_n}{\sqrt 2 u_n + w_n} \right) \cup
\left(\frac{\sqrt 2 t_n + v_n}{\sqrt 2 u_n + w_n}, \frac{t_n}{u_n}\right).$$
By the assumption that $p/q \neq t_n/u_n, v_n/w_n$ and \eqref{bbb2}, we have 
$$ \frac pq \notin \left[ \frac{v_n}{w_n}, \frac{t_n}{u_n} \right].
$$
Then, by Lemma~\ref{lem3}, we conclude that $p/q$ is not an $\mathbf H_4$-best approximation of $\alpha$.
\end{proof}

From now on, we denote by $m=m(\alpha)$ the integer $m$ such that $d_i=3$ for $1\le i\le m$ and $d_{m+1}\not=3$, where $\alpha=[d_1,d_2,\dots]$, as above.

\begin{remark}
Lemma~\ref{lem:BA2} does not hold for general Hecke group $\mathbf H_k$.
For an example, if we consider $k = 5$, then $\lambda_5 = \frac{\sqrt 5 + 1}{2}$.
One can construct the $\mathbf H_5$-expansion in the same manner we introduced in Section~\ref{sec:AR}.
Note that $R = \begin{pmatrix}0&1\\-1&\lambda_5\end{pmatrix}$, thus we have an expansion with four letters since $R^5=I$.
Let $\alpha \in \left( \frac{\lambda_5}{2\lambda_5+1}, \frac{\lambda_5}{\lambda_5+2} \right)$.
Then we have 
$$G_1 = \begin{pmatrix} 1 & 0 \\ \lambda_5 & 1 \end{pmatrix} \quad \text{ and } \quad  
G_2 = \begin{pmatrix} \lambda_5 & \lambda_5 \\ \lambda_5 +2 & 2 \lambda_5 + 1 \end{pmatrix}.
$$
However, if $\alpha \in \left( \frac{\lambda_5}{2\lambda_5+1}, \frac{2}{3\lambda_5} \right)$, then $\frac{1}{2\lambda_5}$ is an $\mathbf H_5$-best approximation of $\alpha$. 
Indeed, there is no matrix of $\begin{pmatrix} 1 & v \\ 2\lambda_5 & w \end{pmatrix}$ or $\begin{pmatrix} t & 1 \\ u & 2\lambda_5 \end{pmatrix}$ in $\mathbf H_5$ satisfying that $\alpha \in \left( \frac{v}{w}, \frac{1}{2\lambda_5}\right)$ or $\alpha \in \left( \frac{1}{2\lambda_5}, \frac{t}{u}\right)$.
One can see it using Ford circles of $\mathbf H_5$. 
Thus, $G_n\cdot0$, $G_n\cdot\infty$ are not enough to characterize $\mathbf H_5$-best approximations.
\end{remark}

Let
\begin{equation}\label{eq:HJ}
H = \begin{pmatrix} -1 & 0 \\ 0 & 1 \end{pmatrix}, \qquad
J = \begin{pmatrix} 0 & 1 \\ 1 & 0 \end{pmatrix}, 
\qquad
d^\vee = \begin{cases}
1 , & \text{ if } \ d = 3, \\
2 , & \text{ if } \ d = 2, \\
3 , & \text{ if } \ d = 1.
\end{cases}
\end{equation}
Then we have that
\begin{equation}\label{MJ}
A_d^{-1} = HA_dH, \qquad 
A_{d^\vee} = JA_dJ. 
\end{equation}
Since $G_n = \begin{pmatrix} t_n & v_n \\ u_n & w_n \end{pmatrix}$,
therefore, for $\alpha = [d_1,d_2,\dots] \in [0, \infty]$, we have 
$$
A_{d_n} \cdots A_{d_2} A_{d_1} = H A^{-1}_{d_n} \cdots A^{-1}_{d_2} A^{-1}_{d_1} H = H G^{-1}_n H= \begin{pmatrix} w_n & v_n \\ u_n & t_n \end{pmatrix}.   %%d
$$
Let     %%d
$$
\alpha^*_n  = \frac{w_n}{u_n} = [d_n, d_{n-1}, \dots, d_1, 3^\infty]  %%d
$$
%\alpha_n = [d_{n+1}, d_{n+2}, \dots ], \qquad 
for $n \ge 0$.
Then 
$$
G_n^{-1} \cdot \alpha = \alpha_n, \qquad
G_n^{-1} \cdot \infty = - \alpha^*_n.
$$
Therefore $G_n^{-1}$ sends the geodesic connecting $\infty$ and $\alpha$ to the geodesic connecting $-\alpha^*_n$ and $\alpha_n$.
Note that $\alpha_n^*\not=1$.
See Figure~\ref{fig:geod}.

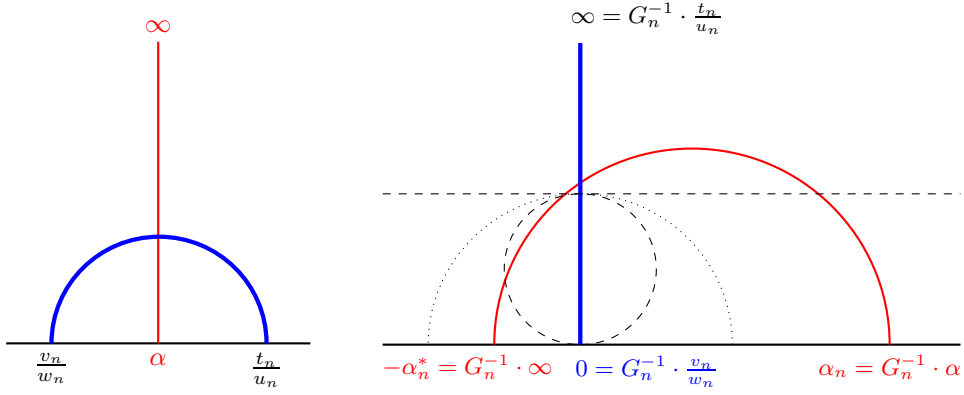
\begin{figure}
\centering
\begin{tikzpicture}[scale=40]
	\draw[thick] (2.46, 0) -- (2.56, 0);
    \draw[thick, red] (2.51,.1) -- (2.51, 0) node [below] {$\alpha$};
    \node[above, red] at (2.51,.1) {$\infty$};
    \node[below] at ({7/(2*sqrt(2))},0) {$\frac{v_n}{w_n} $};
%    \draw[ultra thick] ({7/(2*sqrt(2))},0) node [below] {$\frac{b_n}{s_n} $} arc (-90:270:{1/16}); 
    \node[below] at ({(9*sqrt(2)/5)},0) {$\frac{t_n}{u_n}$};
%    \draw[ultra thick] ({(9*sqrt(2)/5)},0) node [below] {$\frac{a_n}{c_n}$} arc (-90:270:{1/50}); 

    \draw[ultra thick, blue] ({7/(2*sqrt(2))},0) arc (180:0:{1/(20*sqrt(2))});
\end{tikzpicture}
\qquad
\begin{tikzpicture}[scale=2]
	\draw[thick] (-1.3, 0) -- (2.5, 0);

    \draw[thick, red] ({-2*sqrt(2)/5},0) arc (180:0:1.3)  node [below] {\small $\alpha_n = G_n^{-1} \cdot \alpha$};
    \node[below, red, xshift=-10pt] at ({-2*sqrt(2)/5},0) {\small $-\alpha^*_n = G_n^{-1} \cdot \infty$}; %{$-\frac{2\sqrt 2}{5}$};
    
    \draw[ultra thick, blue] (0,2) -- (0,0) node [below, xshift=25pt] {\small $0 = G_n^{-1} \cdot \frac{v_n}{w_n}$};
	\node[above, xshift=25pt] at (0,2) {\small $\infty = G_n^{-1} \cdot \frac{t_n}{u_n}$};

    \draw[dashed] (-1.3, 1) -- (2.5, 1);
    \draw[dashed] (0,0) arc (-90:270:0.5); 

    \draw[dotted] (1,0) arc (0:180:1); 

%    \node[below] at (1,0) {1};
%    \node[below] at (-1,0) {-1};

\end{tikzpicture}
\caption{The action of $G^{-1}_n$ on the upper half plane $\mathbb H$. 
The dashed circles are Ford circles based at 0 and infinity, and the dotted line is the geodesic with end points $-1$ and $1$.}
\label{fig:geod}
\end{figure}

\begin{lemma}\label{lem:BA1}
Let $\alpha = [d_1, d_2, \dots ]$
be a positive real number not in $\mathbb Q(\mathbf H_4)$.
For $n \ge m(\alpha)+1$, we have the following statements.
\begin{enumerate}[label=(\roman*)]
\item $t_n/u_n$ is an $\mathbf H_4$-best approximation of $\alpha$ if and only if $\alpha_n > 1$ or $\alpha^*_n > 1$.

\item $v_n / w_n$ is an $\mathbf H_4$-best approximation of $\alpha$ if and only if $\alpha_n < 1$ or $\alpha^*_n < 1$.
\end{enumerate}
\end{lemma}

\begin{proof}
We prove the case (i) only since the proof for the case (ii) is similar.

Suppose that $\alpha^*_n  = \frac{w_n}{u_n}> 1$.
Let us consider $p/q \in\mathbb Q (\mathbf H_4)$ such that $p/q \ne t_n/u_n$ and $q \le u_n$.
Then $p/q \not=v_n/w_n$.
Combining with Lemma~\ref{lem1}, we deduce that $p/q$ is not in the closed interval $[\frac{t_n}{u_n}, \frac{v_n}{w_n}]$. 
By Lemma~\ref{lem3}, we have $|u_n \alpha - t_n| < | q\alpha - p|$.
Therefore, $t_n/u_n$ is a best approximation.

Suppose that $\alpha_n \in (1, \infty)$, which is equivalent to that 
$\alpha \in G_n\cdot(1,\infty)=(\frac{t_n + v_n}{u_n + w_n}, \frac{t_n}{u_n})$.
Let us consider $p/q \in\mathbb Q (\mathbf H_4)$ such that $q \le u_n$ and $p/q \ne t_n/u_n$.
By Lemma~\ref{lem1}, $p/q$ is not in the open interval $( \frac{v_n}{w_n}, \, \frac{t_n}{u_n} )$. 
If $p/q\not=v_n/w_n$,
then by Lemma~\ref{lem3}, 
we have $|u_n\alpha -t_n| < | q \alpha - p|$.
If $p/q=v_n/w_n$, then $\frac{t_n + v_n}{u_n + w_n} < \alpha$
implies that
$$
| u_n \alpha - t_n | = t_n - u_n \alpha  < w_n \alpha - v_n = | q \alpha - p|.
$$
Therefore, $t_n/u_n$ is an $\mathbf H_4$-best approximation.

Suppose that $u_n \ge w_n $ and 
$\alpha \in (\frac{v_n}{w_n}, \frac{t_n + v_n}{u_n + w_n}]$.
Then  
$$
0 < w_n \alpha - v_n \le t_n - u_n \alpha.
$$
Therefore, we deduce that $t_n/u_n$ is not an $\mathbf H_4$-best approximation.
\end{proof}
See Figure~\ref{fig:regions} for the case of $w_n < u_n$.
In this case, $\frac{v_n}{w_n}$ is an $\mathbf H_4$-best approximation for $\alpha \in \big(\frac{v_n}{w_n}, \frac{t_n}{u_n}\big)$.
Moreover, $\frac{t_n}{u_n}$ is an $\mathbf H_4$-best approximation for $\alpha \in \big(\frac{t_n+v_n}{u_n+w_n}, \frac{t_n}{u_n}\big)$.
By Lemma~\ref{lem:BA2} and \ref{lem:BA1}, we have the following proposition.

\begin{figure}
\centering
\begin{tikzpicture}[scale=6]
	\draw[thick] (-.6, 0) -- (1.4, 0);

	\draw[ultra thick] (0,0) arc (-90:10:0.7071); 
	\draw[ultra thick] (0,0) arc (270:170:0.7071); 

	\draw[ultra thick] (1,0) arc (-90:270:0.35355); 

	\draw[ultra thick] (1/2,0) %node [below, xshift=-3pt] {\small $\frac{b_{n+1}}{d_{n+1}}$} 
    arc (-90:270:0.08839); 
	\draw[ultra thick] (2/3,0) %node [below, xshift=3pt] {\small $\frac{a_{n+1}}{c_{n+1}}$} 
    arc (-90:270:0.07857); 

    \draw (0, 0.02) -- +(0, -0.04) node[below] {$\frac{v_n}{w_n}$};
    \draw (1, 0.02) -- +(0, -0.04) node[below] {$\frac{t_n}{u_n}$};
    \draw ({2-sqrt(2)}, 0.02) -- +(0, -0.04) node[below] {$\frac{t_n+v_n}{u_n+w_n}$};
    \draw[<->, blue] (1, -0.18) -- (0, -0.18) node[left] {$\frac{v_n}{w_n}$};
    \draw[<->, red] ({2-sqrt(2)}, -0.26) -- (1, -0.26) node[right] {$\frac{t_n}{u_n}$};
\end{tikzpicture}
\caption{The regions for $\alpha$ such that $\frac{v_n}{w_n}$ and $\frac{t_n}{u_n}$ are $\mathbf H_4$-best approximations for the case $\alpha^*_n = w_n/u_n >1$}
\label{fig:regions}
\end{figure}
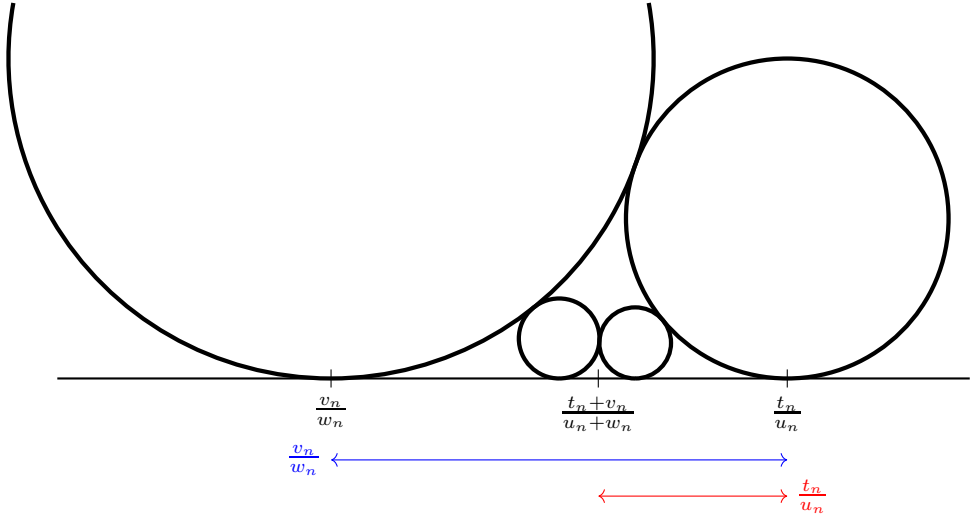

\begin{proposition}\label{prop0}
Let $\alpha = [d_1, d_2, \dots ]$ be a positive real number not in $\mathbb Q (\mathbf H_4)$
and $m=m(\alpha)$, i.e., $\alpha \in (m \sqrt 2, (m+1) \sqrt 2)$.
The set of $\mathbf H_4$-best approximations of $\alpha$ is 
\begin{equation}\label{eq:BA}
\left\{ \frac{t_n}{u_n} \in \mathbb Q (\mathbf H_4) \, | \, \alpha_n > 1 \text{ or } \alpha^*_n > 1, \ n \ge m+1 \right\} 
\cup \left\{ \frac{v_n}{w_n} \in \mathbb Q (\mathbf H_4) \, | \, \alpha_n < 1 \text{ or } \alpha^*_n < 1, \ n \ge m+1  \right\}. 
\end{equation}
\end{proposition}

\section{Proof of Theorem \ref{thm:DS}}\label{sec:DS}

One can rearrange the elements of the set of $\mathbf H_4$-best approximations in \eqref{eq:BA} in order of the denominators. 
%By \eqref{uw} and the analogous relation of $t_n$ and $v_n$, we have $p=p'$ if $q=q'$ for $p/q$, $p'/q'$ in the set of \eqref{eq:BA}.
We denote by $\left\{ \frac{p_i}{q_i} \right\}_{i=1}^\infty$ the sequence of $\mathbf H_4$-best approximations of $\alpha$, arranged in order of the denominators, i.e., $q_{i} < q_{i+1}$ for all $i \in \mathbb N$.
The following lemma describes expressions of $p_i/q_i$ in terms of $t_n/u_n$ and $v_n/w_n$.

\begin{lemma}\label{llll}
Let $\frac{p_i}{q_i}$ be an $\mathbf H_4$-best approximations of $\alpha$ 
satisfying that 
$\frac{p_i}{q_i} = \frac{t_n}{u_n} \ne \frac{t_{n+1}}{u_{n+1}}$,
$n \ge m(\alpha)+1$.
Then we have 
$$
\frac{p_{i+1}}{q_{i+1}} = 
\begin{cases}
\dfrac{v_n}{w_n}, & \text{ if } \ \alpha^*_n >1, \ \alpha_n < 1, \\
\dfrac{t_{n+1}}{u_{n+1}}, & \text{ if } \ \alpha^*_n >1, \ \alpha_n \in [1, \sqrt 2), \\
\dfrac{v_{n+1}}{w_{n+1}}, & \text{ if } \ \alpha^*_n <1, \ \alpha_n \in (1, \sqrt 2).
\end{cases} 
$$
Similarly, if
$\frac{p_i}{q_i} = \frac{t_n}{u_n} \ne \frac{t_{n+1}}{u_{n+1}}$,
$n \ge m+1$,
then we have
$$
\frac{p_{i+1}}{q_{i+1}} = 
\begin{cases}
\dfrac{t_n}{u_n}, & \text{ if } \ \alpha^*_n < 1, \ \alpha_n > 1, \\
\dfrac{v_{n+1}}{w_{n+1}}, & \text{ if } \ \alpha^*_n < 1, \ \alpha_n \in \Big(\frac{1}{\sqrt2}, 1 \Big], \\
\dfrac{t_{n+1}}{u_{n+1}}, & \text{ if } \ \alpha^*_n > 1, \ \alpha_n \in \Big(\frac{1}{\sqrt2}, 1 \Big).
\end{cases} 
$$
\end{lemma}

\begin{proof}
Suppose that $\frac{p_i}{q_i} = \frac{t_n}{u_n}$, $n \ge m+1$.
We may assume that $\frac{t_n}{u_n} > \frac{t_{n+1}}{u_{n+1}}$.
Then $d_{n+1} \in \{ 1,2\}$ and we have $\alpha_n = [d_{n+1}, d_{n+2}, \dots ] < \sqrt 2$. 
Thus Proposition~\ref{prop0} implies that there are three cases as follows: 

(i) If $\alpha^*_n > 1$ and $\alpha_n < 1$, then $d_{n+1} \in \{ 2, 3\}$, thus by \eqref{uw}, we have %%d
$$
u_n < w_n \le \min \{ u_{n+1} , w_{n+1} \}
$$
by Proposition~\ref{prop0}, $\frac{v_n}{w_n}$ is an $\mathbf H_4$-best approximation. 
Therefore, we deudce   %%d
$\frac{p_{i+1}}{q_{i+1}} = \frac{v_n}{w_n}$.

(ii) If $\alpha^*_n > 1$ and $\alpha_n \in [1, \sqrt 2)$, then 
$d_{n+1} = 2$, thus by \eqref{uw}, we have %%d
$$
u_n < w_n < u_{n+1} < w_{n+1} \quad \text{ and } \quad \alpha^*_{n+1} = \frac{w_{n+1}}{u_{n+1}}> 1.
$$
By Proposition~\ref{prop0}, $\frac{v_n}{w_n}$ is not an $\mathbf H_4$-best approximation and $ \frac{p_{i+1}}{q_{i+1}} = \frac{t_{n+1}}{u_{n+1}}$.

(iii) If $\alpha^*_n < 1$ and $\alpha_n \in (1, \sqrt 2)$, then 
$d_{n+1} = 2$, thus by \eqref{uw}, we have %%d
$$
w_n < u_n < w_{n+1} < u_{n+1} \quad \text{ and } \quad \alpha^*_{n+1} = \frac{w_{n+1}}{u_{n+1}} < 1.
$$
By Proposition~\ref{prop0}, the next $\mathbf H_4$-best approximation is  %%d
$\frac{p_{i+1}}{q_{i+1}} = \frac{v_{n+1}}{w_{n+1}}$.

The case of $\frac{p_i}{q_i} = \frac{v_n}{w_n}$ is symmetric.
\end{proof}

\begin{proposition}\label{Ubound}
We have 
$$
\frac 12  < q_{i+1}|q_{i} \alpha - p_{i} | < \frac{\sqrt 2 + 1}{2}.
$$
\end{proposition}

\begin{proof}
Since $\alpha=G_n\cdot \alpha_n$ as in  \eqref{Galphan}, we have
\begin{equation}\label{bound*}
|u_n \alpha - t_n | 
= \left| \frac{u_n(t_n \alpha_{n} + v_n)}{u_n \alpha_n + w_n} - t_n \right|
= \frac{|u_n v_n - w_n t_n|}{u_n \alpha_{n} + w_n}
= \frac{1}{u_n \alpha_n + w_n}. 
\end{equation}
We consider the case $\frac{p_i}{q_i} = \frac{t_n}{u_n} \ne \frac{t_{n+1}}{u_{n+1}}$.
Using Lemma~\ref{llll}, we distinguish three cases: \\ 
(i) If $\frac{p_{i+1}}{q_{i+1}} = \frac{v_n}{w_n}$, then $\alpha^*_n > 1$ and $\alpha_n < 1$.
Then we have
\begin{equation*}\label{b1}
q_{i+1}|q_{i} \alpha - p_{i} | = w_n |u_n \alpha - t_n | 
= \frac{w_n}{u_n \alpha_n + w_n} = \frac{\alpha^*_n}{ \alpha_n + \alpha^*_n },
\end{equation*}
thus
$$
\frac{1}{2} < q_{i+1}|q_{i} \alpha - p_{i} | < 1.
$$
(ii) If $\frac{p_{i+1}}{q_{i+1}} = \frac{t_{n+1}}{u_{n+1}}$, then $\alpha^*_n > 1$ and $\alpha_n \in [1,\sqrt 2)$. 
Then we have $d_{n+1} = 2$, which implies
\begin{equation}\label{b2}
q_{i+1}|q_{i} \alpha - p_{i} | = u_{n+1}|u_n \alpha - t_n | 
= \frac{\sqrt 2 u_n + w_n}{u_n \alpha_n + w_n} 
= \frac{\sqrt 2 + \alpha^*_n}{\alpha_n + \alpha^*_n}.
\end{equation}
Thus,
$$
1 < q_{i+1}|q_{i} \alpha - p_{i} | < \frac{\sqrt 2 + 1}{2}.
$$
(iii) If $\frac{p_{i+1}}{q_{i+1}} = \frac{v_{n+1}}{w_{n+1}}$, then $\alpha^*_n < 1$ and $\alpha_n \in (1,\sqrt 2)$.
Thus we have $d_{n+1} = 2$, which implies
\begin{equation*}\label{b3}
q_{i+1}|q_{i} \alpha - p_{i} | = w_{n+1} \left| u_n \alpha - t_n \right| 
= \frac{u_n + \sqrt 2 w_n}{u_n \alpha_n + w_n} 
= \frac{1 + \sqrt 2 \alpha^*_n}{\alpha_n + \alpha^*_n}.
\end{equation*}
Thus
$$
\frac 1{\sqrt 2} < q_{i+1}|q_{i} \alpha - p_{i} | < \frac{\sqrt 2 + 1}{2}.
$$

The proof for the case $\frac{p_i}{q_i} = \frac{v_n}{w_n}$ is symmetric.   %%d
\end{proof}

The proof of the following lemma is similar to the proof of the classical Diophantine approximation property on $\mathbf H_3$.

\begin{lemma}\label{lem:UA}
We have 
$$
\min_{\substack{p/q \in \mathbb Q (\mathbf H_4) \\ 1 \le q < q_{i}}} | q \alpha - p | = | q_{i-1} \alpha - p_{i-1} |. 
$$
\end{lemma}

\begin{proof}
Suppose that there exists $p/q \in \mathbb Q (\mathbf H_4)$ such that $1 \le q < q_i$ and $| q\alpha - p | < | q_{i-1} \alpha - p_{i-1} |$.
Since $p_{i-1}/q_{i-1}$ is an $\mathbf H_4$-best approximation, we have $q_{i-1}<q$ 
and $q$ is minimal with these properties.   %%d
%We can choose $p/q$ as the smallest such $q$.  %%d
Then $p/q$ is an $\mathbf H_4$-best approximation, which contradicts that $\{ p_i/q_i\}$ is the set of $\mathbf H_4$-best approximations.
\end{proof}

\begin{proof}[Proof of Theorem~\ref{thm:DS}]
Let $N$ be a real number.
Choose $i\in \mathbb N$ as $q_{i-1} \le N < q_i$.
Then, by Lemma~\ref{lem:UA} and Proposition~\ref{Ubound}, we have
$$
\min_{\substack{p/q \in \mathbb Q (\mathbf H_4) \\ 1 \le q \le N}} | q \alpha - p | = | q_{i-1} \alpha - p_{i-1} | < \frac{\sqrt 2 + 1}{2} \cdot \frac{1}{q_i} < \frac{\sqrt 2 + 1 }{2} \cdot \frac 1N.
$$
Therefore, \eqref{thmua} is satisfied with $p/q = p_{i-1}/q_{i-1}$.
\end{proof}

The proof of the following lemma is similar to the classical Diophantine approximation on $\mathbf H_3$.
Recall that $K(\alpha)$ is defined in \eqref{DTSRC}.

\begin{lemma}\label{lem:K}
We have
$$
K(\alpha) = \limsup_{i \to \infty} q_i | q_{i-1} \alpha - p_{i-1}|.
$$
\end{lemma}

\begin{proof}
Let 
$$
S= \limsup_{i \to \infty} q_i | q_{i-1} \alpha - p_{i-1}|.
$$
Then, by Lemma~\ref{lem:UA}, for any $\epsilon >0$, 
for a sufficiently large $N$ with $q_{i-1} \le N < q_i$,
we have
$$
\min_{\substack{p/q \in \mathbb Q (\mathbf H_4) \\ 1 \le q \le N}} | q \alpha - p | = | q_{i-1} \alpha - p_{i-1} | < \frac{S +\epsilon}{q_i} < \frac{S+\epsilon}{N}.
$$
Therefore, \eqref{DTSRC} is satisfied for $p/q = p_{i-1}/q_{i-1}$ with $K = S + \epsilon$.
Hence, $K(\alpha) \le S$.

On the other hand, for any $\epsilon>0$,
there exist arbitrary large $i$'s such that
$$
q_i | q_{i-1} \alpha - p_{i-1}| > S - \epsilon.
$$
Choose $N = q_{i-1}$. 
Then, by Lemma~\ref{lem:UA},
$$
\min_{\substack{p/q \in \mathbb Q (\mathbf H_4) \\ 1 \le q \le N}} | q \alpha - p | = | q_{i-1} \alpha - p_{i-1} | > \frac{S-\epsilon}{q_i} = \frac{S-\epsilon}{N}.
$$
Therefore, \eqref{DTSRC} has no solution for $K = S -\epsilon$,
which implies  %%d is followed by 
$K(\alpha) \ge S$.
Hence, $K(\alpha) = S$.
\end{proof}

\begin{example}\label{example:sup}
(1) Let $\alpha = 1  = [ 2, 2, 2, \dots ]$. 
Then $G_n = A_2^n$ and $\alpha_{n} = [2,2,2,\dots] 
= 1$ for all $n \ge 1$.
Since $\alpha^*_n = [2,2,\dots, 2, 3^\infty] >1$ and $\alpha^*_n \to 1$ as $n \to \infty$, 
by \eqref{b2}, we have
$$
\lim_{i \to \infty}  q_{i+1}|q_{i} \alpha - p_{i} |
= \lim_{n \to \infty} u_{n+1} | u_n \alpha - t_n|
=\lim_{n \to \infty} \frac{\sqrt 2 u_n + w_n}{u_n + w_n} = \lim_{n \to \infty} \frac{\sqrt 2 + \alpha^*_n}{1+\alpha^*_n} = \frac{\sqrt 2 + 1}{2}.
$$
By Lemma~\ref{lem:K},
$$
K(\alpha) = \frac{\sqrt 2 + 1}{2}.
$$

%\noindent 

(2) 
For 
$$
\alpha = \frac{3 + \sqrt{17}}{2\sqrt{2}}= [(323121)^\infty] = [3,2,3,1,2,1,3, \dots ],
$$
we have $m = 1$ (which is defined in Lemma~\ref{lem:BA2}) and $\alpha_2, \alpha_6>1$, $\alpha_3, \alpha_4, \alpha_5 < 1$, $\alpha^*_2, \alpha^*_3> 1$ and $\alpha^*_4, \alpha^*_5, \alpha^*_6< 1$. We have $\alpha_{i}=\alpha_{i+6}$ and $\alpha^*_{i}=\alpha^*_{i+6}$.
Since 
\begin{gather*}
%G_1 = A_3 = \begin{pmatrix} 1 & \sqrt 2 \\ 0 & 1 \end{pmatrix}, \quad     %%d
G_2 = 
%A_3 A_2 =     %%d
\begin{pmatrix} 2\sqrt 2 & 3  \\ 1 & \sqrt 2 \end{pmatrix}, \quad
G_3 = 
%A_3 A_2 A_3 =    %%d
\begin{pmatrix} 2\sqrt 2 & 7  \\ 1 & 2\sqrt 2 \end{pmatrix}, \quad
G_4 
= \begin{pmatrix} 9\sqrt 2 & 7  \\ 5 & 2\sqrt 2 \end{pmatrix}, \\  
G_5 = \begin{pmatrix} 25 & 16\sqrt 2 \\ 9\sqrt 2 & 9 \end{pmatrix}, \qquad 
G_6 = \begin{pmatrix} 57 & 16\sqrt 2 \\ 16\sqrt 2 & 9 \end{pmatrix},
\end{gather*} 
we have the sequence of $\mathbf H_4$-best approximations as
$$
\frac{t_2}{u_2} = \frac{t_3}{u_3} = \frac{2\sqrt 2}{1}, \quad  \frac{v_3}{w_3} = \frac{v_4}{w_4} = \frac{7}{2\sqrt 2}, \quad  \frac{v_5}{w_5} =\frac{v_6}{w_6} = \frac{16\sqrt 2}{9}, \quad \frac{t_6}{u_6} =\frac{57}{16\sqrt 2}, \ \dots. 
$$
\end{example}

\section{Rosen and dual Rosen continued fraction expansion}\label{sec:RCF}

In this section, we discuss the relations between the $\mathbf H_4$-best approximations and the convergents of Rosen and dual Rosen continued fractions.

Let us define the Rosen continued fraction map $f : \left( \sqrt 2,\infty \right) \setminus \mathbb Q(\mathbf H_4) \to \left(\sqrt 2,\infty \right) \setminus \mathbb Q(\mathbf H_4)$ by
\begin{equation}\label{exp_R}
f(x) = \frac{1}{|x-a_0 \sqrt{2}|} \quad \text{ where $a_0(x)$ satisfies }  \  a_0 \sqrt 2 - \frac{1}{\sqrt 2} < x < a_0 \sqrt 2 + \frac{1}{\sqrt 2}
\end{equation}
and let $\epsilon_1(x) \in \{ -1, +1\}$ be the sign of $x - a_0(x) \sqrt 2$.
Define $a_i = a_0(f^i(x))$ and $\epsilon_{i+1} = \epsilon_1(f^i(x))$ for $i \ge 1$.
Then the Rosen continued fraction of $x$ is given by
$$
x = \left \llbracket a_0 ; \epsilon_1 / a_1,  \epsilon_2/a_2, \dots \right \rrbracket \quad \text{ with } \ a_i+\epsilon_{i+1}\ge 1 \text{ for } i\ge0.
$$
We note that 
$f \left( \llbracket a_0 ; \epsilon_1 / a_1,  \epsilon_2/a_2, \dots  \rrbracket \right) = 
\left \llbracket a_1 ; \epsilon_2 / a_2,  \epsilon_3/a_3, \dots \right \rrbracket.$
See also \cites{BKS00, DKS09} for the Rosen continued fraction map on the interval which is conjugate of $f$. 

The expansion of the Rosen continued fraction is given by composition of the matrices $A_1J$, $A_2$, $A_3$ corresponding to the linear fractional maps
\begin{equation}\label{eq:AB}
A_1J \cdot \alpha = JA_3 \cdot \alpha = \frac{1}{\sqrt 2 + \alpha}, \qquad
A_2 \cdot \alpha = \sqrt 2 - \frac{1}{\sqrt 2 + \alpha}, \qquad
A_3 \cdot \alpha = \sqrt 2 + \alpha.
\end{equation}
where $J=\left( \begin{smallmatrix}0&1\\1&0\end{smallmatrix} \right)$ defined as in \eqref{eq:HJ}.
We modify $G_n$ to express it as a product of $A_1J$, $A_2$, $A_3$.
Let 
\begin{equation}\label{DefM}
M_n  
= 
\begin{cases} G_n, &\text{ if } \ \alpha^*_n > 1, \\ 
G_nJ, &\text{ if } \ \alpha^*_n < 1, 
\end{cases}
\quad \text{ and } \quad
\alpha'_n = \begin{cases}
\alpha_n, & \text{ if } \alpha^*_n > 1, \\
J \cdot \alpha_n, & \text{ if } \alpha^*_n < 1.
\end{cases}
\end{equation}
Thus, $\alpha = M_n \cdot \alpha'_n$.
Note that $M_n\cdot\infty$ chooses the one with the smaller denominator between $t_n/u_n$ and $v_n/w_n$.
We have the following recurrence relation.

\begin{lemma}\label{RCF}
For each $n \ge 0$, we have 
\begin{equation}\label{eq:RCF}
\alpha'_{n} = \begin{cases}
A_1J \cdot \alpha'_{n+1}, & \text{ if } \ \alpha'_{n} \in \big(0, \frac{1}{\sqrt 2}\big), \\
A_2 \cdot \alpha'_{n+1},& \text{ if } \ \alpha'_{n} \in \big(\frac{1}{\sqrt 2}, \sqrt 2\big), \\
A_3 \cdot \alpha'_{n+1},& \text{ if } \ \alpha'_{n} \in \big(\sqrt 2, \infty\big).
\end{cases}
\end{equation}
\end{lemma}

\begin{proof}
We note that that $\alpha^*_1 < 1$ if and only if $d_1 = 1$.
Therefore, the lemma holds for $n=0$.
Since $\alpha_n = A_{d_{n+1}} \cdot \alpha_{n+1}$,
by \eqref{MJ}, we have for $n\ge 0$
\begin{equation*}
\alpha'_n =\begin{cases} A_{d_{n+1}} \cdot \alpha'_{n+1},&\text{if }\alpha^*_n>1,\ \alpha^*_{n+1}>1,\\
A_{d_{n+1}}J \cdot \alpha'_{n+1},&\text{if }\alpha^*_n>1,\ \alpha^*_{n+1}<1,\\
JA_{d_{n+1}} \cdot \alpha'_{n+1} = A_{d_{n+1}^\vee}J \cdot \alpha'_{n+1},&\text{if }\alpha^*_n<1,\ \alpha^*_{n+1}>1,\\
JA_{d_{n+1}}J \cdot \alpha'_{n+1}=A_{d_{n+1}^\vee}\cdot \alpha'_{n+1},&\text{if }\alpha^*_n<1,\ \alpha^*_{n+1}<1.
\end{cases}
\end{equation*}
If $\alpha'_{n} \in \big(0, \frac{1}{\sqrt 2}\big)$, then 
$d_{n+1}=1$, $\alpha^*_n > 1$ or $d_{n+1}=3$, $\alpha^*_n < 1$, thus $\alpha^*_{n+1}<1$, $\alpha^*_n > 1$ or $\alpha^*_{n+1}>1$, $\alpha^*_n < 1$.
If $\alpha'_{n} \in \big(\frac{1}{\sqrt 2}, \sqrt 2 \big)$, then $d_{n+1}=2$, thus $\alpha^*_{n+1}<1$, $\alpha^*_n < 1$ or $\alpha^*_{n+1}>1$, $\alpha^*_n > 1$.
If $\alpha'_{n} \in \big(\sqrt 2, \infty\big)$, then 
$d_{n+1}=3$, $\alpha^*_n > 1$ or $d_{n+1}=1$, $\alpha^*_n < 1$, thus $\alpha^*_{n+1}>1$, $\alpha^*_n > 1$ or $\alpha^*_{n+1}<1$, $\alpha^*_n < 1$.
\end{proof}

\begin{proposition}\label{prop1}
Let $\alpha$ be a positive real number not in $\mathbb Q (\mathbf H_4)$ and $m=m(\alpha)$.
The set of the convergents of the Rosen continued fraction expansion of $\alpha$ satisfies that   
$$
\left\{ \frac{r_i}{s_i} \in \mathbb Q (\mathbf H_4) \, | \, i \ge 0 \right\} =
\left\{ \frac{t_n}{u_n} \in \mathbb Q (\mathbf H_4) \, | \, \alpha^*_n > 1, \ n \ge m+ 1 \right\} \cup \left\{ \frac{v_n}{w_n} \in \mathbb Q (\mathbf H_4) \, | \, \alpha^*_n < 1, \ n \ge m+ 1 \right\}.
$$
\end{proposition}

\begin{proof}
Let $\alpha \not\in \mathbb Q(\mathbf H_4)$ such that $\alpha\in(0,\infty)$.
Let $\alpha$ have the Rosen continued fraction expansion
$$\alpha = \left \llbracket a_0 ; \epsilon_1 / a_1, \epsilon_2/a_2, \dots \right \rrbracket.$$
Note that
$a_i \ge 2$ if $\epsilon_{i+1} = -1$ for $i \ge 1$.
Let
\begin{equation}\label{eq:ABp}
\varphi \left( \llbracket a_0 ; \epsilon_1 / a_1,  \epsilon_2/a_2, \dots  \rrbracket \right)
= \begin{cases}
\llbracket a_1 ; \epsilon_2 / a_2,  \epsilon_3/a_3, \dots  \rrbracket, &\text{if } a_0 = 1, \epsilon_1 = +1, \\
\llbracket a_1 ; \epsilon_2 / a_2,  \epsilon_3/a_3, \dots  \rrbracket, &\text{if } a_0 = 2, \epsilon_1 = -1, \\
\llbracket a_0 -1 ; \epsilon_1 / a_1,  \epsilon_2/a_2, \dots  \rrbracket, &\text{otherwise}.
\end{cases}
\end{equation}
Then, from \eqref{eq:AB}, we have
\begin{equation}\label{eq:ABp1}
\varphi^n \left(\alpha + \sqrt 2 \right) = \alpha'_n + \sqrt 2.
\end{equation}
If $\epsilon_1=+1$, then $m=a_0$ and $d_{m+1}=1$, and if $\epsilon_1=-1$, then $m=a_0-1$ and $d_{m+1}=2$.
Thus $\alpha = A_3^mA_1J \cdot \alpha'_{m+1}$ if $\epsilon_1=+1$, and $\alpha =A_3^mA_2 \cdot \alpha'_{m+1}$ if $\epsilon_1=-1$.
From \eqref{eq:ABp1}, we have
\begin{equation}\label{eq:M_m+1}
\alpha'_{m+1} = \llbracket a_1-1;\epsilon_2/a_2,\epsilon_3/a_3,\dots\rrbracket
\quad \text{ and } \quad f \left( \alpha + \sqrt 2\right) = \alpha'_{m+1} + \sqrt 2 = \varphi^{m+1} \left( \alpha + \sqrt 2\right).
\end{equation}
Let $m_i = m(\llbracket a_i;\epsilon_{i+1}/a_{i+1},\epsilon_{i+1}/a_{i+1},\dots\rrbracket)$ for $i\ge0$, i.e.,
$$m_i = \begin{cases} a_i, &\text{if } \ \epsilon_{i+1}=+1,\\
a_i-1, &\text{if } \ \epsilon_{i+1}=-1.\end{cases}$$
Then $m_0\ge 0$ and $m_i\ge 1$ for $i\ge 1$.
For each $n\ge 1$, there exist $i\ge0$ and $b$ such that
$$n = m_0+m_1+\cdots+m_{i-1}+ b, \quad  1\le b \le m_i.$$
By \eqref{eq:M_m+1}, we have 
$$
\alpha'_n = \llbracket a_i- b; \epsilon_{i+1}/a_{i+1},\epsilon_{i+2}/a_{i+2},\dots\rrbracket
\quad\text{and}\quad
\alpha = M_n\cdot \alpha'_n = a_0\sqrt{2}+\dfrac{\epsilon_1}{\ddots+\dfrac{\ddots}{a_{i-1}\sqrt{2}+\dfrac{\epsilon_i}{b \sqrt 2 +\alpha'_n}}},$$
thus $M_n\cdot \infty = \llbracket a_0;\epsilon_1/a_1,\dots,\epsilon_{i-1}/a_{i-1}\rrbracket = \frac{r_{i-1}}{s_{i-1}}$.
Therefore, the set of convergent of the Rosen continued fraction $\frac{r_i}{s_i}$ for $i \ge 0$ is
$$
\left \{ \frac{r_i}{s_i} \in \mathbb Q (\mathbf H_4)  \, | \, i \ge 0 \right \} = \left\{ M_n \cdot \infty \, | \,  n \ge m + 1 \right\}.
$$
From \eqref{DefM}, we have 
\begin{align*}
\left\{ M_n \cdot \infty \, | \,  n \ge m + 1 \right\} &= \left\{ G_n \cdot \infty \, | \, \alpha^*_n > 1, \  n \ge m + 1 \right\} \cup \left\{ G_n \cdot 0 \, | \,  \alpha^*_n < 1, \ n \ge m + 1 \right\} \\
&= \left\{ \frac{t_n}{u_n} \in \mathbb Q (\mathbf H_4) 
\, | \, \alpha^*_n > 1, \ n \ge m+ 1 \right\} \cup \left\{ \frac{v_n}{w_n} \in \mathbb Q (\mathbf H_4) 
\, | \, \alpha^*_n < 1, \ n \ge m+ 1 \right\}.
\end{align*}
\end{proof}

Let us define the dual Rosen continued fraction map $\tilde f: [1, \infty) \setminus \mathbb Q (\mathbf H_4) \to [1, \infty) \setminus \mathbb Q (\mathbf H_4)$ by
\begin{equation}\label{exp_dR}
\tilde{f}(x)=\frac{1}{\left|x- \tilde a_0 \sqrt{2}\right|},\quad \text{ where $\tilde a_0 = \tilde a_0(x)$ satisfies } \ (\tilde a_0-1)\sqrt{2}+1\le \alpha< \tilde a_0\sqrt{2}+1.
\end{equation}
Let $\tilde \epsilon_1(x) \in \{ -1, +1\}$ be the sign of $x- \tilde a_0 \sqrt{2}$.
Define $\tilde a_i(x) = \tilde a_0 \left (\tilde f^i (x)\right)$ and 
$\tilde \epsilon_{i+1}(x) = \tilde \epsilon_{1} \left (\tilde f^i (x)\right)$ for $i \ge 1$.
Then we have the dual Rosen continued fraction expansion
$$x = \llbracket \tilde a_0; \tilde\epsilon_1 / \tilde a_1, \tilde\epsilon_2 / \tilde a_2 , \dots \rrbracket,$$
where $\tilde a_i\ge 1$ and $\tilde a_i \ge 2$ if $\tilde \epsilon_i = -1$.
Let 
\begin{equation}\label{eq:dRCF}
N_n = 
\begin{cases}
G_n, 
& \text{if }  \alpha_n > 1 \text{ or } \alpha_n = 1,  \alpha^*_n > 1, \\
G_nJ, 
& \text{if }  \alpha_n < 1 \text{ or } \alpha_n = 1,  \alpha^*_n < 1,
\end{cases} \quad
\tilde \alpha'_n = 
\begin{cases}
\alpha_n, 
& \text{if }  \alpha_n > 1 \text{ or } \alpha_n = 1,  \alpha^*_n > 1, \\
J \cdot \alpha_n, 
& \text{if }  \alpha_n < 1 \text{ or } \alpha_n = 1,  \alpha^*_n < 1.
\end{cases}
\end{equation}
Then $\alpha = N_n \cdot \tilde \alpha'_n$ and 
we have the following recurrence relation.

\begin{lemma}\label{dRCF}
Let $[d_1, d_2, \dots ]$ be the $\mathbf H_4$-expansion of $\alpha$.
We have for $n \ge 0,$
\begin{equation}\label{eq:rec_dRCF}
\tilde \alpha'_n = \begin{cases}
A_2 \cdot \tilde \alpha'_{n+1},& \text{ if } \ \tilde \alpha'_n \in [1, \sqrt 2), \\
JA_1 \cdot \tilde \alpha'_{n+1}, & \text{ if } \ \tilde \alpha'_n \in  (\sqrt 2,\sqrt 2 +1), \\
A_3 \cdot \tilde \alpha'_{n+1},& \text{ if } \ \tilde \alpha'_n \in [\sqrt 2 +1, \infty).
\end{cases}
\end{equation}
\end{lemma}

\begin{proof}
Suppose $\tilde \alpha'_{n} \in \big[1, \sqrt 2 \big)$.
Then $\alpha_n \in \big( \frac{1}{\sqrt 2}, \sqrt 2 \big)$ and $d_{n+1}=2$.
If $\alpha_n > 1$, then $\alpha_{n+1}>1$, 
if $\alpha_n < 1$, then $\alpha_{n+1}<1$, 
if $\alpha_n = 1$, then $\alpha_{n+1}=1$ and  
$\alpha^*_{n} > 1$ is equivalent to $\alpha^*_{n+1} > 1$.
Therefore, by \eqref{MJ} and \eqref{eq:dRCF}
$$
\tilde \alpha'_n = A_2 \cdot \tilde \alpha'_{n+1}.
$$
If $\tilde \alpha'_{n} \in \big(\sqrt 2, \sqrt 2 +1 \big)$, then $\alpha_{n} \in \big(\sqrt 2, \sqrt 2 +1 \big) = A_3 \cdot (0,1)$, which implies $\alpha_{n+1} \in (0,1)$, or 
$\alpha_{n} \in \big(\frac{1}{\sqrt 2+1}, \frac{1}{\sqrt 2} \big)= A_1 \cdot (1,\infty)$, which implies $\alpha_{n+1} \in (1, \infty)$.
Therefore, by \eqref{MJ} and \eqref{eq:dRCF}
$$
\tilde \alpha'_n = J A_1 \cdot \tilde \alpha'_{n+1}.
$$
If $\tilde \alpha'_{n} \in \big[\sqrt 2 +1, \infty\big)$, 
then $\alpha_{n} \in \big[\sqrt 2 +1, \infty\big) = A_3 \cdot [1, \infty)$, which implies $\alpha_{n+1} \in [1, \infty)$, $\alpha^*_{n+1}>1$ or 
$\alpha_{n} \in \big(0, \frac{1}{\sqrt 2+1} \big]= A_1 \cdot (0,1]$, which implies $\alpha_{n+1} \in (0,1]$, $\alpha^*_{n+1}<1$.
Therefore, by \eqref{MJ} and \eqref{eq:dRCF}
\[
\tilde \alpha'_n = A_3 \cdot \tilde \alpha'_{n+1}. \qedhere
\]
\end{proof}

By Lemma~\ref{dRCF}, $N_n$ is a product of the matrices $JA_1$, $A_2$, $A_3$ corresponding to the linear fractional maps
$$
JA_1 \cdot \alpha = A_3 J \cdot \alpha = \sqrt 2 + \frac{1}{\alpha}, \qquad
A_2 \cdot \alpha = \sqrt 2 - \frac{1}{\sqrt 2 + \alpha}\qquad \text{and}\qquad
A_3 \cdot \alpha = \sqrt 2 + \alpha.$$
Note that
\begin{equation}\label{dualRCF}
\left \llbracket \tilde a_0 ; \tilde \epsilon_1 / \tilde a_1,  \tilde \epsilon_2/\tilde a_2, \dots \right \rrbracket
= \begin{cases}
A_2 \cdot \left \llbracket  \tilde a_1 -1 ; \tilde \epsilon_2 / \tilde a_2, \tilde \epsilon_3/\tilde a_3, \dots\right \rrbracket , &\text{ if } \tilde a_0 = 1, \tilde \epsilon_1 = -1, \\ 
JA_1 \cdot \left \llbracket \tilde a_1 ; \tilde \epsilon_2 / \tilde a_2, \tilde \epsilon_3/a_3, \dots\right \rrbracket , &\text{ if } \tilde a_0 = 1, \tilde \epsilon_1 = +1, \\ 
A_3 \cdot \left \llbracket \tilde a_0 -1 ; \tilde \epsilon_1 / \tilde a_1, \tilde \epsilon_2/\tilde a_2, \dots \right \rrbracket, &\text{ otherwise}.
\end{cases}
\end{equation}

\begin{proposition}\label{prop2}
Let $\alpha$ be a positive real number not in $\mathbb Q (\mathbf H_4)$. 
Then we have 
$$
\left \{ \frac{\tilde r_i}{\tilde s_i} \, | \, i \ge 1 \right \}
\subset 
\left\{ \frac{t_n}{u_n} 
\ | \ \alpha_n > 1, \ n \ge m+1 \right\} \cup \left\{ \frac{v_n}{w_n} 
\ | \ \alpha_n < 1, \ n \ge m+1 \right\}
\subset \left \{ \frac{\tilde r_i}{\tilde s_i} \, | \, i \ge 1 \right \} \cup \left \{ \frac{r_0}{s_0} \right \}.
$$
\end{proposition}

\begin{proof}
Let $\alpha \in [1, \infty)$ be a real number not in $\mathbb Q (\mathbf H_4)$ with the dual Rosen continued fraction expansion
$$\alpha = \llbracket \tilde a_0; \tilde\epsilon_1 / \tilde a_1, \tilde\epsilon_2 / \tilde a_2 , \dots \rrbracket,$$
where $\tilde a_i \ge 1$ and $\tilde a_i \ge 2$ if $\tilde \epsilon_i = -1$ for $i \ge 1$.
We check $\alpha \in \big[ (\tilde a_0 -1 ) \sqrt 2 + 1, \tilde a_0 \sqrt 2 + 1 \big)$.
By Lemma~\ref{dRCF}, we have
\begin{equation*}
\alpha  =
\begin{cases}
 A_3^{\tilde a_0 -1}A_2 \cdot \tilde \alpha'_{\tilde a_0},&\text{if }\alpha \in \big[ (\tilde a_0 -1 ) \sqrt 2 + 1, \tilde a_0 \sqrt 2 \big),\\
A_3^{\tilde a_0-1}JA_1 \cdot \tilde \alpha'_{\tilde a_0},&\text{if }\alpha \in \big( \tilde a_0 \sqrt 2, \tilde a_0 \sqrt 2 + 1 \big).
\end{cases}
\end{equation*}
Thus, using \eqref{dualRCF}, we have
\begin{equation}\label{eq:N_a}
\tilde \alpha'_{\tilde{a}_0} =
\begin{cases}
\llbracket \tilde{a}_1-1;\tilde \epsilon_2 /\tilde{a}_2,\tilde \epsilon_3/\tilde{a}_3,\dots\rrbracket,&\text{if } \ \tilde{\epsilon}_1=-1,\\
\llbracket \tilde{a}_1;\tilde \epsilon_2/\tilde{a}_2,\tilde \epsilon_3/\tilde{a}_3,\dots\rrbracket,&\text{if } \ \tilde{\epsilon}_1=+1.
\end{cases}
\end{equation}

Let 
$$
\tilde m_0 = \tilde a_0 \quad \text{ and } \quad
\tilde m_i = \begin{cases} \tilde a_i -1, & \text{ if } \tilde \epsilon_{i} = -1, \\ 
\tilde a_i, & \text{ if } \tilde \epsilon_{i} = +1,
\end{cases} \qquad \text{ for } \ i \ge 1.
$$
We note that $\tilde m_i \ge 1$ for $i \ge 0$.  
For each $n \ge 1,$ we have $i \ge 0$ and $\tilde b$ satisfying that 
$$n = \tilde m_0 + \tilde m_1 + \dots + \tilde m_{i} - \tilde b \ \text{ and } \ 1 \le \tilde b \le \tilde m_i.$$
By \eqref{eq:N_a}, we have
$$
\tilde \alpha'_n = \llbracket \tilde b ; \tilde \epsilon_{i+1} / \tilde a_{i+1},  \tilde \epsilon_{i+2}/ \tilde a_{i+2}, \dots \rrbracket, \quad 
\alpha = N_n\cdot \tilde \alpha'_n = \tilde a_0\sqrt{2}+\dfrac{\tilde \epsilon_1}{\ddots+\dfrac{\ddots}{\tilde a_{i-1}\sqrt{2}+\dfrac{\tilde \epsilon_i}{(\tilde a_i - \tilde b)\sqrt 2 +\tilde \alpha'_n}}}.
$$
Therefore, we have 
$$
N_n \cdot \infty = 
\left \llbracket \tilde a_0 ; \tilde \epsilon_1 / \tilde a_1,  \epsilon_2 / \tilde a_2, \dots , \tilde \epsilon_{i-1} / \tilde a_{i-1} \right \rrbracket = \frac{\tilde r_{i-1}}{\tilde s_{i-1}}.
$$
Recall that $\frac{\tilde r_{i-1}}{\tilde s_{i-1}}$ is the $(i-1)$-th convergent of the dual Rosen continued fraction of $\alpha$.

By \eqref{exp_R} and \eqref{exp_dR},
if $a_0 = \frac{r_0}{s_0} \ne \frac{\tilde r_0}{\tilde s_0} = \tilde a_0$, then $\alpha \in \left( m\sqrt 2 + \frac{1}{\sqrt 2}, m\sqrt 2 +1 \right)$ for some $m \ge 1$.
Then we have the dual Rosen continued fraction expansion
$$
\alpha = \left \llbracket m ; +1/ \tilde a_1,  \epsilon_2 / \tilde a_2, \dots \right \rrbracket \ \text{ with } \
\tilde \alpha'_m=\left \llbracket  \tilde a_1 ; \tilde \epsilon_2 / \tilde a_2,  \epsilon_3 / \tilde a_3, \dots \right \rrbracket \in (1, \sqrt 2).
$$
Therefore,
$$
\tilde a_0 = m = \tilde m_0, \quad \tilde \epsilon_1 = +1, \quad \tilde a_1 = 1 = \tilde m_1, \quad \tilde \epsilon_2 = -1.
$$
If $n \ge m +1 = \tilde m_0 + \tilde m_1$, then we have  
$$
N_n \cdot \infty = \frac{\tilde r_i}{\tilde s_i} \quad \text{ for some } i \ge 1.
$$
Hence, we have 
$$
\left\{ N_n \cdot \infty \, | \,  n \ge m +1 \right\} = \left \{ \frac{\tilde r_i}{\tilde s_i} \, | \, i \ge 1 \right \}
\quad \text{ or } \quad 
\left\{ N_n \cdot \infty \, | \,  n \ge m +1 \right\} = \left \{ \frac{\tilde r_i}{\tilde s_i} \, | \, i \ge 1 \right \}
\cup \left \{ \frac{r_0}{s_0} \right \}.
$$
By \eqref{eq:dRCF}, we have 
\begin{align*}
\left\{ N_n \cdot \infty \, | \,  n \ge m +1 \right\} 
&= \left\{ G_n \cdot \infty \, | \, \alpha_n > 1, \  n \ge m +1 \right\} \cup \left\{ G_n \cdot \infty \, | \, \alpha_n = 1, \ \alpha^*_n > 1, \ n \ge m +1 \right\} \\
&\quad \cup \left\{ G_n J \cdot \infty \, | \, \alpha_n < 1, \ n \ge m +1 \right\} \cup \left\{ G_n J \cdot \infty \, | \, \alpha_n = 1, \ \alpha^*_n < 1, \ n \ge m +1 \right\} \\
&= \left\{ \frac{t_n}{u_n} \, | \, \alpha_n > 1, \ n \ge m +1 \right\} \cup \left\{ \frac{t_n}{u_n} \, | \, \alpha_n =1, \ \alpha^*_n > 1, \ n \ge m +1 \right\} \\
&\quad \cup \left\{ \frac{v_n}{w_n} \, | \, \alpha_n < 1, \ n \ge m +1 \right\} \cup \left\{ \frac{v_n}{w_n} \, | \, \alpha_n =1, \ \alpha^*_n < 1, \ n \ge m +1 \right\}. 
\end{align*}
Therefore, we complete the proof.
\end{proof}

\begin{proof}[Proof of Theorem~\ref{thm:BA}]
It is a direct consequence of Propositions \ref{prop0}, \ref{prop1} and \ref{prop2}.
\end{proof}

\section{Properties of the 
\texorpdfstring{$\mathbf H_4$}{H4}-best approximations}
\label{sec:LT}

We first give the bound of asymptotic approximation for the $\mathbf H_4$-best approximations.

\begin{theorem}\label{thm:AAR}
Let $p/q$ be an $\mathbf H_4$-best approximation of $\alpha$. 
Then we have
$$
\left| \alpha - \frac{p}{q} \right| < \frac{1}{q^2}.
$$
Moreover, if $p/q$ is a convergent of the Rosen continued fraction and the dual Rosen continued fraction, then
\begin{equation*}\label{eq:both}
\left| \alpha - \frac{p}{q} \right| < \frac{1}{2q^2}.
\end{equation*}
If $p/q$ is not a convergent of the Rosen continued fraction or $p/q$ is not a convergent of the dual Rosen continued fraction, 
then 
\begin{equation}\label{eq:one}
\frac{1}{(\sqrt 2 + 1) q^2} < \left| \alpha - \frac{p}{q} \right| < \frac{1}{q^2}.
\end{equation}
\end{theorem}

\begin{proof}
We first consider the case that $\frac{p}{q} = \frac{t_n}{u_n}$.
By Proposition~\ref{prop0}, $\alpha^*_n>1$ or $\alpha_n > 1$.
By \eqref{bound*}, we have 
\begin{equation*}
q |q \alpha - p | = u_n|u_n \alpha - t_n | 
= \frac{u_n}{u_n \alpha_n + w_n} = \frac{1}{\alpha_n + \alpha^*_n} < 1.
\end{equation*}
By Propositions \ref{prop1} and \ref{prop2}, 
if $\frac{p}{q}$ is a convergent of the Rosen continued fraction and the dual Rosen continued fraction, then $\alpha^*_n > 1$ and $\alpha_n > 1$.
\begin{equation*}
q |q \alpha - p | = u_n|u_n \alpha - t_n | = \frac{1}{\alpha_n + \alpha^*_n} < \frac 12.
\end{equation*}
Suppose that $\frac{p}{q}$ is not a convergent of the Rosen continued fraction.
Then choose $n$ as $\frac{p}{q} = \frac{t_n}{u_n} \ne \frac{t_{n+1}}{u_{n+1}}$.
Thus $d_{n+1} \in \{ 1,2\}$ and Proposition \ref{prop1} implies $\alpha^*_n < 1$ and $\alpha_n \in (1, \sqrt 2)$. 
Therefore,
\begin{equation*}
\frac{1}{\sqrt 2 + 1} < u_n|u_n \alpha - t_n | 
= \frac{1}{\alpha_n + \alpha^*_n} < 1.
\end{equation*}
When $\frac{p}{q}$ is not a convergent of the dual Rosen continued fraction, choose $n$ as $\frac{p}{q} = \frac{t_n}{u_n} \ne \frac{t_{n-1}}{u_{n-1}}$.
Thus $d_{n} \in \{ 1,2\}$ and Proposition \ref{prop2} implies $\alpha^*_n \in (1, \sqrt 2)$ and $\alpha_n < 1$. 
Therefore,
\begin{equation*}
\frac{1}{\sqrt 2 + 1} < u_n|u_n \alpha - t_n | 
= \frac{1}{\alpha_n + \alpha^*_n} < 1.
\end{equation*}

Suppose that $\frac{p}{q} = \frac{v_n}{w_n}$.
Then $\alpha^*_n < 1$ or $\alpha_n < 1$.
By $\alpha=G_n\cdot \alpha_n$, we have 
\begin{equation}\label{eq:rs2}
q |q \alpha - p | = w_n | w_n \alpha - v_n | 
= \frac{w_n \alpha_n}{u_n \alpha_n + w_n} = \frac{1}{(\alpha^*_n)^{-1} + (\alpha_n)^{-1}} < 1.
\end{equation}
If $\frac{p}{q}$ is a convergent of the Rosen continued fraction and the dual Rosen continued fraction, then $\alpha^*_n < 1$ and $\alpha_n < 1$.
\begin{equation*}
q |q \alpha - p | = w_n|w_n \alpha - v_n | = \frac{1}{(\alpha^*_n)^{-1} + (\alpha_n)^{-1}} < \frac 12.
\end{equation*}
If $\frac{p}{q}=\frac{v_n}{w_n}$ is not a convergent of Rosen continued fraction or not a convergent of the dual Rosen continued fraction, then, in a similar way to case $\frac{p}{q}=\frac{t_n}{u_n}$, we have
\begin{equation*}
\frac{1}{\sqrt 2 + 1} < u_n|u_n \alpha - t_n | 
= \frac{1}{(\alpha^*_n)^{-1} + (\alpha_n)^{-1}} < 1.
\end{equation*}
\end{proof}

Let $\alpha = [d_1, d_2, \dots]$ be a real number given by
$$
d_n = \begin{cases}
3, & \text{ if } \ 4^i \le n < 2\cdot 4^i, \\ 
2, & \text{ if } \ 2\cdot 4^i \le n < 3\cdot 4^i, \\ 
1, & \text{ if } \ 3\cdot 4^i \le n < 4^{i+1}. 
\end{cases}
$$
Let
$$
n_i = 3\cdot 4^i - 2, \qquad n'_i = 2 \cdot 4^i -1.
$$
Then
$\alpha_{n_i}, \alpha_{n'_i} \in \big( \frac{1}{\sqrt 2}, 1 \big)$ and $\alpha^*_{n_i}, \alpha^*_{n'_i} >1$. 
$$
\lim_{i \to \infty} \alpha^*_{n_i} = 1, \quad 
\lim_{i \to \infty} \alpha_{n_{i}} = \frac{1}{\sqrt 2}, 
\quad
\lim_{i \to \infty} \alpha^*_{n'_i} = \infty, \quad 
\lim_{i \to \infty} \alpha_{n'_{i}} = 1.
$$
Therefore, by \eqref{eq:rs2}, we have
$$
\lim_{i\to\infty} w_{n_i} \left | w_{n_i} \alpha - v_{n_i} \right| = \frac{1}{\sqrt 2 + 1} \quad \text{ and } \quad 
\lim_{i\to\infty} w_{n'_i} \left | w_{n'_i} \alpha - v_{n'_i} \right| = 1.
$$
It implies that the constants in \eqref{eq:one} are optimal.
Note that for the case $p/q$ is the convergent of the Rosen continued fraction, Theorem~\ref{thm:AAR} was shown in \cite{Na95}*{Proposition 12}.

Next we state and prove Legendre's theorem associated with $\mathbf H_4$-best approximations.
\begin{theorem}
If 
\begin{equation}\label{bounm_cond}
\left|\alpha-\frac{p}{q}\right|<\frac{1}{2q^2},
\end{equation}
then $p/q$ is an $\mathbf H_4$-best approximation of $\alpha$.
Moreover, the constant 1/2 is the optimal value.
\end{theorem}

\begin{proof}
Suppose that there exists $r/s \ne p/q$ in $\mathbb Q (\mathbf H_4)$ satisfying
$$|s\alpha-r|\le |q\alpha-p|<\frac{1}{2q}.$$
Then,
$$\left|\alpha-\frac{r}{s}\right|<\frac{1}{2sq}.$$
Thus,
$$\left|\frac{r}{s}-\frac{p}{q}\right|\le\left|\alpha-\frac{r}{s}\right|+\left|\alpha-\frac{p}{q}\right|<\frac{1}{2sq}+\frac{1}{2q^2}=\frac{s+q}{2sq^2}.$$
On the other hand, 
$$\left|\frac{r}{s}-\frac{p}{q}\right| = \frac{|rq - sp|}{sq} \ge \frac{1}{sq}.$$
Therefore, 
$$\frac{1}{sq}<\frac{s+q}{2sq^2}.$$
Thus, $q<s$.
Thus, $p/q$ is an $\mathbf H_4$-best approximation.

Let $\alpha= [ d_1, d_2, d_3, \dots ]$ be the real number given by
$$
d_{n} = \begin{cases} 3, &\text{ if } \ n = 3^i \ \text{ for some } \ i \ge 0, \\
2, &\text{ otherwise}.
\end{cases}
$$
Then $\alpha_{n} > 1$ and $\alpha^*_{n} > 1$ for all $n \ge 1$.
Therefore, $\frac{v_n}{w_n}$ is not an $\mathbf H_4$-best approximation for $n \ge 1$. %and
%$\frac{r_j}{s_j} = \frac{t_{j+1}}{u_{j+1}}$ for $j \ge 1$.
Moreover, for $n_i = 2 \cdot 3^i$, we deduce that
$$  
\lim_{i \to \infty} \alpha_{n_i} = 1, 
\qquad
\lim_{i \to \infty} \alpha^*_{n_i} = 1.
$$
We have a subsequential limit of \eqref{eq:rs2} along $\{n_i\}$ as
\begin{equation*}
\lim_{i \to \infty} w_{n_i}|w_{n_i} \alpha - v_{n_i} | 
= \lim_{i \to \infty} \frac{1}{(\alpha^*_{n_i})^{-1} + (\alpha_{n_i})^{-1}} = \frac 12.
\end{equation*}
Therefore, the constant 1/2 in \eqref{bounm_cond} cannot be replaced by any larger number.
\end{proof}

\begin{remark}
Combined with \eqref{eq:one}, we deduce that if
$$
\left|\alpha-\frac{p}{q}\right|<\frac{1}{(\sqrt 2 + 1) q^2},
$$
then $p/q$ is a convergent of the Rosen continued fraction.
This is stated in \cite{Na95}*{Proposition 13}.
\end{remark}

\section*{Acknowledgement}

The authors wish to thank Yann Bugeaud for his helpful comments.   %%d

A.B. acknowledges the support by Centro di Ricerca Matematica Ennio de Giorgi, Scuola Normale Superiore di Pisa and UniCredit Bank R\&D division for financial support  under the project `Dynamics and Information Research Institute-Quantum Information (Teoria dell’Informazione), Quantum Technologies'. 
A.B. would like to thank Centro de Giorgi for the excellent working conditions and the research travel support.
D.K.~was supported by the National Research Foundation of Korea (NRF-2018R1A2B6001624, RS-2023-00245719).
%and the Dongguk University Research Fund of 2022.
S.L. was supported by the Institute for Basic Science (IBS-R003-D1) and BK21 SNU Mathematical Sciences Division.

\Addresses

\end{document}